\documentclass[11pt]{amsart}
\usepackage{amssymb}
\usepackage{wasysym}
\usepackage[colorlinks,linkcolor={blue}]{hyperref}

\usepackage[colorinlistoftodos]{todonotes}
\makeatletter
\providecommand\@dotsep{5}
\makeatother
\def \a{\alpha}

\def \ga{\gamma}
\def\dga{{\dot{\gamma}}}
\def \Ga{\Gamma}
\def \de{\delta}
\def \e{\varepsilon}
\def\deta{{\dot{\eta}}}
\def \la{\lambda}
\def \La{\Lambda}

\def\dxi{{\dot{\xi}}}

\def \vr{\varphi}

\def\vrt{\vartheta}
\def\th{\theta}

\def \om{\omega}
\def \Om{\Omega}
\def\Si{\Sigma}

\def \re{{\mathbb R}}
\def \na{{\mathbb N}}
\def \Z {{\mathbb Z}}

\def \lv{\left\vert}
\def \rv{\right\vert}
\def \lV{\left\Vert}
\def \rV{\right\Vert}

\def\bB{{\mathbf B}}
\def\B{{\mathbb B}}
\def\E{{\mathbb E}}
\def\bE{{\mathbf E}}
\def\cE{{\mathcal E}} 
\def\F{{\mathbb F}}
\def\cF{{\mathcal F}} 
\def\H{{\mathbb H}}
\def\cH{{\mathcal H}} 
\def\cK{{\mathcal K}}
\def\K{{\mathbb K}}
\def\cL{{\mathcal L}} 
 
\def\N{{\mathbb N}}

\def\cN{{\mathcal N}} 

\def\P{{\mathbb P}}

\def\bbS{{\mathbb S}}

\def\V{{\mathbb V}}
 
\def\dw{{\dot{w}}}
\def\dx{{\dot{x}}}
\def\X{{\mathbb X}}

\def\HT{{\cH^1_{[0,T]}}}
\def\0{\{0\}}
\def\b0{{\mathbf 0}}

\DeclareMathOperator{\graph}{Graph}
\DeclareMathOperator{\image}{Image}
\DeclareMathOperator{\isom}{Isom}
\DeclareMathOperator{\rank}{rank}

\theoremstyle{plain}

 \newtheorem{MainThm}{Theorem}
 \newtheorem{MainCor}[MainThm]{Corollary}

 \newtheorem{mainthm}{Theorem}
 \newtheorem{maincor}[mainthm]{Corollary}

\swapnumbers

\newtheorem{Thm}{\bf Theorem}[section]

\newtheorem{Lemma}[Thm]{\bf Lemma}

\newtheorem{Corollary}[Thm]{\bf Corollary}

\newtheorem{Proposition}[Thm]{\bf Proposition}

\theoremstyle{remark}
\newtheorem{Definition}[Thm]{\bf Definition}

\title{Minimization and hyperbolicity}
\author{Gonzalo Contreras}
\thanks{Gonzalo Contreras was partially supported by {\sc conacyt}, Mexico, grant A1-S-10145.}

\address{CIMAT \\ P.O. Box 402 \\36.000 Guanajuato GTO \\M\'exico.}
\email{gonzalo@cimat.mx}

\author{Daniel Offin}
\address{Department of Mathematics and Statistics\\ Queen's
University\\ Kingston, Canada\\ K7L 4V1}
\email{offind@mast.queensu.ca}

\begin{document}

\maketitle


\section{Introduction}

In this paper we study the relationship between strict locally
minimizing orbits for time dependent lagrangians and hyperbolicity
properties of the corresponding lagrangian flow.

The phase space of a time dependent lagrangian is naturally equipped
with a codimension 1  foliation given by fixing the value of the time
parameter. This foliation is preserved by the lagrangian flow and is
transversal to its vector field. The existence of such foliation
simplifies the arguments in this work. For this reason and for
simplicity of notation we choose to write the proofs in the autonomous
case, giving only the corresponding statements for time dependent
lagrangians.

\subsection{Autonomous Lagrangians.}
\quad

Let $M$ be a closed riemannian manifold and $\pi:TM\to M$ its tangent
bundle. An autonomous {\it lagrangian} on $M$ is a $C^2$ function
$L:TM\to\re$ satisfying the following {\it strict convexity} condition:
\renewcommand{\theenumi}{\alph{enumi}}
\begin{equation}\label{sconvexity}
\exists b>0\qquad
\forall (x,v)\in TM\qquad \forall h\in T_xM\qquad
\tfrac{\partial^2L}{\partial^2
v}(x,v)\cdot(h,h) \ge b\, |h|_x^2.
\end{equation}
Observe that condition~\eqref{sconvexity} implies the following {\it superlinearity} property:
\begin{equation}\label{superlinearity}
\lim_{|v|\to+\infty}\tfrac{L(x,v)}{|v|}=+\infty
\qquad \text{ uniformly on $(x,v)\in TM$.}
\end{equation}

 The {\it Euler-Lagrange equation} which  in local coordinates is 
 \begin{equation}\label{EL}
 \frac{d}{dt}\frac{\partial L}{\partial v}(x,\dx)=
 \frac{\partial L}{\partial x}(x,\dx),
 \end{equation}
 defines a complete flow $\vr$ on $TM$ by $\vr_t(x,v)=(x(t),\dx(t))$,
 where $x(t)$ is a solution of~\eqref{EL} with initial conditions
 $x(0)=x$, $\dx(0)=v$.
 The {\it energy function}
 \begin{equation*}\label{E}
 E(x,v)=\frac{\partial L}{\partial v}\cdot v-L
 \end{equation*}
 is preserved by the Euler-Lagrange flow. By the superlinearity
 property the {\it energy levels} $E^{-1}\{k\}$ are compact
 $\vr_t$-invariant subsets of $TM$.
 
 Given an absolutely continuous curve $\ga:[0,T]\to M$ define the {\it
 action} of $\ga$ by
 $$
 A_L(\ga)=\int_0^TL(\ga(t),\dga(t))\,dt.
 $$
 
 Set 
 $$
 \HT:=\Big\{\,\ga:[0,T]\to M\;\Big|\;\text{$\ga$ is absolutely
 continuous},\,\int_0^T|\dga(s)|^2\,ds<\infty\,\Big\}.
 $$
 Then $\HT$ is a Hilbert manifold whose tangent space at $\ga$ is
 $$
 T_\ga\HT=\Big\{\,\xi:[0,T]\to TM\;\Big|\;
 \pi\circ\xi=\ga,\,\int_0^T|\dxi(s)|^2\,ds<\infty\,\Big\}
 $$
 with the riemannian metric
 $\lV\xi\rV_{\cH^1}:=\lV\xi\rV_{\cL^2}+\Vert\dxi\Vert_{\cL^2}$.
 
 Consider a smooth path $\Ga:[-\e,\e]\to\HT$. Suppose that
 $A_L(\Ga(s))$ is twice differentiable at $s=0$. Differentiating at
 $s=0$ we get the formula for the first variation of the action
 functional:
   \begin{align}
 d_\ga \,A_L\cdot\xi :&= \tfrac{d}{ds}A_L(\Ga(s))\Big\vert_{s=0}
 =\int_0^TL_x\,\xi+L_v\,\dxi\; dt
 \notag\\
 &=\int_0^T\big[L_x-\tfrac{d}{dt}L_v\big]\,\xi\;dt
 + L_v\cdot\xi\,\Big\vert_0^T,
 \label{1var}
 \end{align}
 where $\ga=\Ga(0)$, $\xi$ is the variational vector field
  $\xi(t):=\tfrac{\partial \Ga}{\partial s}(0)(t)$ over $\ga$ and
  $\dxi$ is the covariant derivative of $\xi$ along $\ga$.
  The second variation of the action functional is called the
  {\it index form} and is given by
  \begin{align}\label{index}
  d^2_\ga \,A_L(\xi,\eta):&= I(\xi,\eta)
  :=\int_0^T \big[\dxi\, L_{vv}\,\deta +\dxi\, L_{vx}\,\eta 
  +\xi\, L_{xv}\,\deta  +\xi L_{xx}\eta\big] \;dt,
  \end{align}
  where $L_{xx}$, $L_{xv}$, $L_{vx}$, $L_{vv}$ are bilinear
  operators on $T_{\ga(t)}M$. In local coordinates, these operators
  coincide with the partial derivatives 
  $\frac{\partial L}{\partial v_i \partial v_j}$,
  $\frac{\partial L}{\partial v_i \partial x_j}$,
  $\frac{\partial L}{\partial x_i \partial v_j}$,
  $\frac{\partial L}{\partial v_i \partial v_j}$
  when we use the derivatives
  $\dxi, \deta\in\re^m$ in the local coordinates instead of
  covariant derivatives of the riemannian metric.

  Let $\Si=E^{-1}\{k\}\subset TM$ be an energy level for $L$.
    Let $\La\subset \Si$ be a compact $\vr_t$-invariant set. We say that
  $\La$ is {\it hyperbolic} if there is a splitting
  $T\Si|_\La=E^s\oplus\langle X\rangle\oplus E^u$, where
  $X=\frac{d\;}{dt}\vr_t$ is the vector field of $\vr_t$, such that
  \renewcommand{\theenumi}{\roman{enumi}}
  \begin{enumerate}
  \item\label{Hi} $d\vr_t(E^s(\th))=E^s(\vr_t(\th))$ and 
         $d\vr_t(E^u(\th))=E^u(\vr_t(\th))$ 
	 for all $\th\in \La$, $t\in\re$.
  \item\label{Hii} There exist constants $C,\la>0$, such that
  \begin{align*}
  \lV d\vr_t(\vrt)\rV &\le C\,e^{-\la t} \lV\vrt\rV
  \quad\text{ for all }t\ge 0\text{ and  }\vrt\in E^s,
  \\
  \lV d\vr_{-t}(\vrt)\rV &\le C\,e^{-\la t} \lV\vrt\rV
  \quad\text{ for all }t\ge 0\text{ and  }\vrt\in E^u.
  \end{align*}	 
  
  \end{enumerate} 
 
  Given $\th\in TM$, let $\ga_\th(t):=\pi\circ\vr_t(\th)$ and 
  \begin{equation*}\label{Om}
  \Om(\th,T):=\Big\{\,\xi\in T_{\ga_\th}\HT\;\Big|\;
  \xi(0)=0,\;\xi(T)=0,\;\langle \xi(\tfrac T2),\vr_{\frac
  T2}(\th)\rangle=0 \Big\}.
  \end{equation*}
 
 \begin{MainThm}\label{TA}\qquad
 
 Let $k\in\re$ and $\La\subset E^{-1}\{k\}$ be a compact $\vr_t$-invariant set. If there is
 $a>0$ such that for all $(\th,T)\in\La\times\re^+$ and
 $\xi\in\Om(\th,T)$:
 $$
 d^2_{\ga_\th} A_L(\xi,\xi)\ge a\;\lV\xi\rV^2_{\cL^2}.
 $$
 then $\La$ is hyperbolic.
 \end{MainThm}
 
 Observe that no recurrency assumption is made on the orbits of $\La$.
 Since the lagrangian is $C^2$ and the flow $\vr_t$ is $C^1$, the index form
 $d^2_{\ga_\th} A_L$ is continuous on $(\th,T)$.
 Therefore we get:
 
 \begin{MainCor}\label{CB}\quad
 
 Let $k\in\re$ and $\La\subset E^{-1}\{k\}$ be a compact $\vr_t$-invariant set. If there is a
 sequence of periodic orbits $P_n$ with period $T_n$ such that
 $\lim_n T_n=+\infty$,
 $$
 \La\subset\limsup_n P_n=\{\;\lim_k\th_{n_k}\;\,|\;\,
 \th_n\in P_n,\;\lim_k n_k=\infty,\;\exists\lim_k\th_{n_k}\;\} 
 $$
 and there is $a>0$ such that if $\th_n\in P_n$,
 $\xi\in\Om(\th_n,T_n)$
 $$
 d^2_{\ga_{\th_n}}A_L(\xi,\xi)\ge a\;\lV\xi\rV^2_{\cL^2},
 $$
 then $\La$ is hyperbolic.
 \end{MainCor}
 
 For example it is known (c.f. Ma\~n\'e~\cite{Ma6}) that for generic
 lagrangians the support of minimizing measure can be approximated by periodic
 orbits which minimize the action in their homotopy classes. But we
 don't know if they satisfy the condition on Corollary~\ref{CB}.
 
 We quote now the main result used in the proof of
 Theorem~\ref{TA}. A {\it Jacobi field} $J(t)$ along a  
 solution $\ga$ of the Euler-Lagrange equation~\eqref{EL}
 is the variational vector field
 $J(t)=\frac{\partial f}{\partial s}(0,t)$
 of a smooth variation $]-\e,\e[\ni s\mapsto f(s,t)$
 of $\ga(t)=f(0,t)$ by solutions of~\eqref{EL}.
  We say that a solution $\ga$ of the Euler-Lagrange
 equation~\eqref{EL}  has a {\it conjugate point} at time $T>0$ 
 if there is a non-zero Jacobi field $J$ along $\ga$ with $J(0)=0$ and
 $J(T)=0$. We say that an orbit segment is {\it disconjugate} if it has no
 conjugate points.
  An orbit segment is disconjugate if
 and only if the index form is positive definite\footnote{See formula \eqref{info0}.} on $\Om(\th,T)$.
 In particular minimizing open orbit segments are
 disconjugate.

 Let $\V$ be the {\it vertical subbundle} of $T(TM)$ defined by
 $\V(\th)=\ker d_\th\pi$, $\th\in TM$,  i.e. the tangent spaces
 of the fibers of $TM\to M$. Since the Jacobi fields are the
 projections by $d\pi$ of the derivative of the lagrangian flow, we
 have the following equivalent definition of conjugate points.
  We say that two points $\th_1$, $\th_2$
 are {\it conjugate} if there is $T\in\re$ such that
 $\vr_T(\th_1)=\th_2$ and
 $d\vr_T(\V(\th_1))\cap\V(\th_2)\ne\0$.
 
 If the whole orbit of $\th\in TM$ is disconjugate then
 (c.f.~\cite[Prop. A]{CI}) the following limits exist:
 \begin{align}
 \E(\th):&=\lim_{T\to+\infty} d\vr_{-T}\big(\V(\vr_T(\th))\big),\notag
 \\
 \F(\th):&=\lim_{T\to+\infty} d\vr_{T}\big(\V(\vr_{-T}(\th))\big).
 \label{Green}
 \end{align}
 The subspaces $\E$, $\F$ are called the {\it Green bundles} along the
 orbit of $\th$. Their dimension is $\dim \E=\dim \F=\dim M$ and 
 they always contain the subspace generated by the
 lagrangian vector field $X$.
 
 \begin{MainThm}\label{TC}\quad
 
 Suppose that $k\in\re$ and $\La\subset E^{-1}\{k\}\subset TM$ 
 is a compact $\vr_t$-invariant subset without
 conjugate points. Then $\La$ is hyperbolic if and only if
 $$
 \E(\th)\cap\F(\th)=\langle X(\th)\rangle \text{ for all }\th\in\La.
 $$
 In this case $\E=E^s\oplus\langle X\rangle$, $\F=E^u\oplus\langle X\rangle$.
 \end{MainThm}
 
 This theorem is similar to  Proposition~B in \cite{CI}
 but here we have dropped any recurrency condition on $\La$.
 This result also appears in Johnson, Novo, Obaya~\cite{JNO}
 and Arnaud~\cite{Arnaud4}.

 \subsection{The time-dependent case.}\quad
 
   Let $M$ be a closed riemannian manifold. 
  On the space $C^2(TM,\re)$
  we use the $C^2$ topology on compact subsets.
  Let $\cK$ be a topological space and 
  $\phi:\cK\times\re\to \cK$ a continuous flow on $\cK$.
  A time-dependent {\it lagrangian} on $M\times \cK$ is a function
  $L:TM\times \cK\to\re$ satisfying the following conditions:
  \begin{enumerate}
  \item\label{la1} The restrictions $L\vert_{TM\times\{k\}}$ are $C^2$.
  \item\label{la2} The map $\cK\to C^2(TM,\re)$: $k\mapsto L\vert_{TM\times\{k\}}$
   is continuous.
   \item\label{conv} {\it Strict convexity:} 
   There is $b>0$ such that 
   $\tfrac{\partial^2 L}{\partial v^2}(x,v,k)\cdot (h,h)\ge b \,|h|_x^2$.
 \item\label{comp}{\it Completeness:} The maximal solutions 
 of the {\it
 Euler-Lagrange equation}
 \begin{equation}\label{EL2}
 \frac{d}{dt}\,\frac{\partial L}{\partial v}\big(x,\dx,\phi_t(k)\big)=
 \frac{\partial L}{\partial x}\big(x,\dx,\phi_t(k)\big)
 \end{equation}
 are defined in all $t\in\re$.
  \end{enumerate}
  Observe that the strict convexity~\eqref{conv} implies the {\it superlinearity} property:
  \begin{equation}\label{nasuperlin}
\lim_{|v|\to+\infty}\tfrac{L(x,v,k)}{|v|}=+\infty
\qquad \text{ uniformly on $(x,v,k)\in TM\times \cK$.}
\end{equation}

 Special cases are {\it periodic lagrangians}, where 
 $\cK=S^1=\re/\Z$ and $\phi_t(k+\Z)=(t+k)+\Z$, the
 {\it autonomous lagrangian} where $\cK$ is a point and $\cK=\re$
 with $\phi_t(s)=s+t$.

The lagrangian $L$ defines a complete flow $\vr_t$ on $TM\times \cK$,
 called the {\it lagrangian flow} for $L$ by 
 $\vr_t(x,v,k)=\big(\ga(t),\dga(t),\phi_t(k)\big)$, where
 $s\mapsto\ga(s)$ is a solution of the Euler-Lagrange
 equation~\eqref{EL2} with initial conditions $\ga(0)=x$,
 $\dga(0)=v$.

 Given an absolutely continuous curve $\ga:[0,T]\to M$ and $k\in \cK$,
 define the {\it action} of $(\ga,k)$ by
 $$
 A_L(\ga,k):=\int_0^T L\big(\ga(t),\dga(t),\phi_t(k)\big)\;dt.
 $$
 
 Given a smooth path $\Ga:]-\e,\e[\to\HT$ and fixing $k\in\cK$,
 the formulas for the first and second derivative of $s\mapsto
 A(\Ga(s),k)$ are given by \eqref{1var} and \eqref{index}.

   The lagrangian flow induces a tangent map $d\vr_t$ on
  the bundle $\Pi:T(TM)\times \cK \to TM\times \cK$
  given by $d_{(x,v,k)}\vr_t(\zeta,k)=
  \big( \partial_{(x,v)}f_t(\cdot,\cdot,k)(\zeta),\;\phi_t(k)\big)$,
  where $\vr_t(x,v,k)=\big(f_t(x,v,k),\phi_t(k)\big)$.
  We shall ambiguously identify $d_{(x,v,k)}\vr_t(\zeta,k)$ with its
  projection $\partial_{(x,v)}f_t(\cdot,\cdot,k)(\zeta)$.
  Suppose that $\La\subset TM\times \cK$ is a compact $\vr_t$-invariant
  subset. We say that $\La$ is hyperbolic if there is a splitting
  $\Pi^{-1}(\La)=E^s\oplus E^u$ such that conditions~\eqref{Hi} and
  \eqref{Hii} above hold.
  
  Given $\th\in TM\times\cK$ let $\ga_\th:=\Pi\circ\vr_t(\th)$ and
   \begin{equation*}\label{Om1}
  \Om(\th,T):=\big\{\,\xi\in T_{\ga_\th}\HT\;\,\big|\;\,
  \xi(0)=0,\;\xi(T)=0\, \big\}.
  \end{equation*}

 \begin{mainthm}\quad
 
  Let $\La\subset TM\times \cK$ be a compact invariant set for $\vr_t$.
 If  there is $a>0$ such that for all $(\vrt,T)\in\La\times\re^+$ and
 $\xi\in\Om(\vrt,T)$
 $$
 d_{\ga_\vrt}^2 \,A_L(\xi,\xi)\ge a \;\lV\xi\rV_{\cL^2}^2,
 $$
 then $\La$ is hyperbolic.
 \end{mainthm}

 For the periodic case, we have the following
 
  \begin{maincor}\label{hypcor}\quad
 
 Let $\La\subset TM\times S^1$ 
 be a compact invariant set for $\vr_t$.
 If there is a sequence of periodic orbits
 $P_n$ with period $T_n$ such that
 $\lim_n T_n=+\infty$,
 $\La\subset \limsup_n P_n$
 and there is $a>0$ such that if $\vrt_n\in P_n$,
 $\xi\in\Om(\vrt_n,T_n)$
 $$
 d_{\ga_{\vrt_n}}^2 \,A_L(\xi,\xi)\ge a \;\lV\xi\rV_{\cL^2}^2,
 $$
 then $\La$ is hyperbolic.

\end{maincor}
 
 The {\it vertical bundle} is defined  as $\V=\ker
 d\pi\times\cK\subset T(TM)\times\cK$. Two points $\th_1,\th_2\in
 TM\times\cK$ are {\it conjugate} if there is $T\in\re$ such that
 $\vr_T(\th_1)=\th_2$ and $d\vr_T(\V(\th_1))\cap\V(\th_2)\ne\0$.
 On a disconjugate orbit the Green bundles are defined
 by~\eqref{Green}.
 
 \begin{mainthm}\label{TC1}\quad
 
 If $\La\subset TM\times\cK$ is a compact $\vr_t$-invariant subset without
 conjugate points. Then $\La$ is hyperbolic if and only if
 $\E(\th)\cap\F(\th)=\0 \text{ for all }\th\in\La$.
  In this case $\E=E^s$ and  $\F=E^u$.
 \end{mainthm}

\section{The hamiltonian Jacobi and Riccati equations.}

The {\it Liouville 1-form} $\la=p\,dx$ on the cotangent bundle $T^*M$
is defined by $\la_\th(\xi):=\th(d\pi(\xi))$, where $\xi\in T_\th
T^*M$ and $\pi:T^*M\to M$ is the projection. The {\it symplectic form }
on $T^*M$ is $\om:=d\la=dp\wedge dx$.

The {\it Hamiltonian} $H:T^*M\to\re$ corresponding to the autonomous lagrangian 
\linebreak  
$L:TM\to\re$ is defined by
\begin{equation}\label{hami}
H(x,p):=\max_{v\in T_xM}\big[\, p(v)-L(x,v)\,\big].
\end{equation}
The {\it hamiltonian vector field} $X$ of $H$ is defined by the
equation $\om(X,\,\cdot\,)=-dH$. Its flow $\psi_t$ is called the 
{\it hamiltonian flow} for $H$. The {\it Legendre
transform} $\cF:TM\to T^*M$, defined by $\cF(x,v)=(x,L_v(x,v))$,
maps the  lagrangian vector field of $L$ to the hamiltonian vector field
and hence it conjugates the flows $\vr_t$ and $\psi_t$.

The {\it connection map} $K:T_{(x,v)}TM\to T_xM$ on $TM$ is given by
the covariant derivative $K(\xi)=\frac{D}{ds}\Ga\big|_{s=0}$, where
$\Ga:]-\e,\e[\to TM$ is a smooth curve such that $\Ga'(0)=\xi$. The
map $F:TM\to T^*M$, $F(x,v)=\langle v,\cdot\rangle_x$ induces a
connection map $K^*:= K\circ DF^{-1}:T_\th T^*M\to T_{\pi(\th)}M$ on $T^*M$.
The {\it Sasaki metric} on $T^*M$ is defined by
$$
\langle\xi,\zeta\rangle_\th:=
\langle d\pi(\xi),d\pi(\zeta)\rangle_{\pi(\th)}
+\langle K^*(\xi),K^*(\zeta)\rangle_{\pi(\th)}.
$$

There is a canonical splitting $T_\th T^*M=\H(\th)\oplus \V(\th)$
into two lagrangian subspaces under $\om$, where the {\it vertical
subspace} $\V(\th):=\ker(d_\th \pi)$ is the tangent space to the fiber
$T_x^*M$ and the {\it horizontal subspace} $\H(\th):=\ker K^*$ consist
of tangents to parallel paths $\Ga^*$.
In this splitting the symplectic form $\om=dp\wedge dx$ is written as
\begin{equation}\label{Sform}
\om_\th[(h_1,v_1),(h_2,v_2)]
=\langle v_1,h_2\rangle_{\pi(\th)}
-\langle v_2,h_1\rangle_{\pi(\th)},
\qquad (h_i,v_i)\in \H(\th)\oplus\V(\th).
\end{equation}

In the splitting $T (T^*M)=\H\oplus\V$ the hamiltonian vector field is
given by the {\it hamiltonian equations}
\begin{align}\label{Hamilton}
  \tfrac{dx}{dt}=H_p(x,p),\qquad
  \tfrac{D}{dt}p=-H_x(x,p).
\end{align}
The coordinates $(h(t),v(t))=d\psi_t(h_0,v_0)\in\H\oplus\V$ of the
derivative of the hamiltonian flow satisfy the {\it Jacobi equations:}
\begin{align}\label{Jacobi}
\dot h=H_{px}\,h+H_{pp}\, v\;,
\qquad
\dot v=-H_{xx}\,h-H_{xp}\, v\,,
\end{align}
where the dot means covariant derivative and $H_{xx}$, $H_{xp}$,
$H_{px}$, $H_{pp}$ are linear operators.
The convexity of $L$ implies that $H_{pp}$ is positive definite.
If we change the metric so that the projected orbit $\ga(t)=\pi\circ\psi_{t}(\th)$
becomes a geodesic\footnote{By proposition~\ref{growth}.\eqref{g4} below, 
if the orbit of $\th$ has no conjugate points and the energy level of $\th$ is regular, then $\dga(t)\ne 0$.}
 and we use Fermi coordinates then the linear operators
in \eqref{Jacobi} coincide with the partial derivatives and $\dot h$, $\dot v$
coincide with time derivatives in $\re^n$. Also in these coordinates
$\H=\ker K^*=[\delta p=0]$.

Alternatively we shall use local coordinates $(x,p,t)$, $p=\sum_i p_i dx_i$, along a neighborhood
of an orbit $\{(\psi_t(\th),t):t\in\re\}\subset T^*M\times \re$
with bounded $C^3$ norm.
We use coordinates in $T^*M\times\re$ instead of $T^*M$ in order to split possible self intersections
of the domain of the chart.
In these coordinates the Hamilton and Jacobi equations take the same forms
\eqref{Hamilton} and \eqref{Jacobi}, where $H_x$, $H_p$, $H_{xx}$, $H_{xp}$, $H_{px}$, $H_{pp}$
are partial derivatives and the time derivatives in \eqref{Hamilton} and \eqref{Jacobi}
are usual derivatives in euclidean space. The norm of vectors in these coordinates 
are equivalent to the riemannian norm on the manifold. In these coordinates the symplectic form is 
$\om=dp\wedge dx$,
$$
\om_\th[(h_1,v_1),(h_2,v_2)]= v_1\cdot h_2-v_2\cdot h_1,
$$ and 
$\V(\th) = d\pi(\th)^{-1}\{0\} = \{(0,v): v\in T_{\pi(\th)} M\}$.
In these coordinates we  use
$$
\H(\th):=\{(h,0):h\in T_{\pi(\th)}M\},
$$
which may not coincide with $\ker K^*$ because 
the projected hamiltonian vector field $d\pi X(\th)$ may not be parallel and 
the  coordinate vectors ${\partial}/{\partial x_i}$ may not be parallel
translations along $\pi\circ\psi_t(\th)$.
Nevertheless the subspaces $\H(\th)$, $\V(\th)$
are  lagrangian.

Given $\th\in T^*M$ let  $E\subset T_\th T^*M$ be an $m$-dimensional
subspace, where $m=\dim M$. Suppose that 
\begin{equation}\label{nv}
d_\th\psi_t(E)\cap\V(\psi_t(\th))\ne\0\qquad
\text{ for all } t\in[0,T].
\end{equation}
Then we can write $d\psi_t(E)=\graph(S(t))$, where
$S(t):\H(\psi_t(\th))\to \V(\psi_t(\th))$ is a linear map. This is, if
$\xi\in E$ then
$$
d\psi_t(\xi)=\big( h(t), \,S(t)\, h(t) \big),
$$
where the pair $h(t)$, $v(t)=S(t)\, h(t)$ is a solution of the 
Jacobi equation~\eqref{Jacobi}.
From~\eqref{Sform}, the subspace $E$ is lagrangian iff $S(t)$ is a
symmetric linear map, identifying 
$\H(\vrt)\approx T_{\pi(\vrt)}M\approx \V(\vrt)$ for $\vrt\in T^*M$. 
Replacing this solution in \eqref{Jacobi} we get
$$
\dot S\,h+S\,(H_{px}+H_{pp}\,S)\,h=-H_{xx}\,h-H_{xp}\,S\,h.
$$
Since this holds for all 
$h\in d\pi(d\psi_t(E))= \H(\psi_t(\th))$
we obtain the {\it Riccati equation:}
\begin{equation}\label{riccati}
\dot S+S\,H_{pp}\,S+S\, H_{px}+H_{xp}\,S+H_{xx}=0.
\end{equation}

Our hypothesis~\eqref{nv} implies that there is a matrix solution
$H(t)$, $V(t)$ of the Jacobi equations~\eqref{Jacobi} such that
$d\psi_t(E)=\image\big(H(t),V(t)\big)\subset T_{\psi_t(\th)}(T^*M)$
with $\det H(t)\ne 0$ and $V(t)=S(t)\,H(t)$ for all $t\in[0,T]$.

\bigskip
\bigskip

\section{Conjugate points and  Green bundles.}

Two points $\th_1,\,\th_2\in T^*M$ are said {\it conjugate} if there
is $T\in\re$ such that $\psi_T(\th_1)=\th_2$ and
$d\psi_T(\V(\th_1))\cap \V(\th_2)\ne\0$, i.e. if there is a non-zero
solution $(h(t),v(t))$ of the Jacobi equations~\eqref{Jacobi} along 
the orbit of $\th_1$ with $h(0)=0$ and $h(T)=0$. Since the Legendre
transform satisfies $\pi_{T^*M}\circ\cF=\pi_{TM}$, this definition of
conjugate point corresponds to the one with Jacobi fields given in the introduction.

Let $\th\in T^*M$ and suppose that the orbit segment
$\psi_{[-T,T]}(\th)$ is disconjugate. Define a partial
order on symmetric isomorphisms $A,B:\H(\th)\to\V(\th)$ by $A\prec B$
iff $B-A$ is positive definite. For $t\in]0,T[$ let
$S_t,\,U_t:\H(\th)\to\V(\th)$ be symmetric isomorphisms such that
$E_t:=d\psi_{-t}(\V(\psi_t(\th)))=\graph(S_t)$ and
$F_t:=d\psi_t(\V(\psi_{-t}(\th)))=\graph(U_t)$. By proposition~1.4
in~\cite{CI}, if $0<s<t\le T$ then $S_s\prec S_t\prec U_t\prec U_s$.
When the whole orbit of $\th$ is disconjugate, the
limits $\lim_{t\to+\infty}S_t$ and $\lim_{t\to+\infty}U_t$ exist and
are called the {\it Green subspaces}
\begin{align*}
\E(\th):=\lim_{T\to+\infty}d\psi_{-T}(\V(\psi_T(\th))),
\qquad
\F(\th):=\lim_{T\to+\infty}d\psi_{T}(\V(\psi_{-T}(\th))).
\end{align*}
These subspaces are lagrangian and invariant along the orbit of $\th$.
Moreover, $\E\cap\V=\0$ and $\F\cap\V=\0$.

 \begin{Definition}\label{defbh}\quad
 
  Given a set $\La\subset T^*M\times\re$ and constants $b_1,\,b_2>0$,
 define the set of {\it bounded hamiltonians}
 $\cH(T^*M\times\re,\La,b_1,b_2)$ as the set of functions
 $H:T^*M\times\re\to\re$ such that
 \begin{enumerate}
 \item\label{bh1}
 $H\in C^3(T^*M\times\re,\re)$.
 \item\label{bh3}
 $\lV H\rV_{C^3(\La,\re)}\le b_1$.
 \item\label{bh4}
  $y^* H_{pp}(\th,t)\, y\ge b_2\, |y|_\th^2$ for all
 $(\th,t)\in\La$.              
 \end{enumerate}
 \end{Definition}
  We denote by $\psi_t$ the hamiltonian flow of the bounded
  hamiltonian $H$.

 \begin{Proposition}\label{BdRic}
 Given $b_1,\,b_2>0$ there is $A>0$ such that for any $\La\subset
 T^*M\times\re$ and any bounded hamiltonian 
 $H\in\cH(T^*M\times\re,\La,b_1,b_2)$ if $T>2$,
 $\psi_{[0,T]}(\th)\subset\La$ and
 $S(t):\H(\psi_t(\th))\to\V(\psi_t(\th))$, $t\in]0,T[$ is a symmetric solution of
 the Riccati equation~\eqref{riccati}, then
 $$
 \lV S(t)\rV<A \qquad \text{ for } \quad 1<t<T-1.
 $$
 \end{Proposition}

 Observe that in this proposition $S(t)$ may not be defined at
 $t=0,T$.
 We shall need the following lemma, whose proof is elementary.
 
 \begin{Lemma}
 The function $w(t):=R\, \coth(Rt-d)$, $R>0$, satisfies
 \renewcommand{\theenumi}{\roman{enumi}}
 \begin{enumerate}
  \renewcommand{\theenumi}{\roman{enumi}}
 \item $\dw+w^2-R^2=0$;
 \item $\dw<0$, \; $w(-t+d/R)=-w(t+d/R)$;
 \item $\lim_{t\to+\infty}w(t)=R$,\; $\lim_{t\to-\infty}w(t)=-R$;
 \item $\lim_{t\to(d/R)^+}w(t)=+\infty$,\; 
         $\lim_{t\to(d/R)^-}w(t)=-\infty$.    
 \end{enumerate}
 
 \end{Lemma}

 \noindent{\bf Proof of Proposition~\ref{BdRic}:}
 
  By the spectral theorem, it is enough to prove that $x^* \,S(t)\, x$
 is uniformly bounded for $\lV x\rV=1$ and $1<t<T-1$. 
 Since $H_{px}^* = H_{xp}$ and $S^*=S$, from~\eqref{riccati} we have that
 $$
 \dot{S} + S^*\, H_{pp}\, S + H_{px}^* \, S + S^*\, H_{px} + H_{xx} = 0.
 $$
 Let $C:= (H_{pp})^{-1}\, H_{px}$,\; $D:= H_{xx} - C^*  H_{pp}  C - \dot{C}$
 and $V(t):= S(t) + \, C(\psi_t(\theta))$. Then
 $$
 \dot{V} + V^*\, H_{pp}\, V + D = 0 \, .
 $$
 
 Since by \ref{defbh}.\eqref{bh1} $C$ and $D$ are continuous, they are 
 bounded on $\La$. Hence it is enough to prove that
 $x^* V(t) \, x$ is uniformly bounded on $t\in{]1,T-1[}$.
 Observe that $V(t)$ is not necessarily symmetric. Let
 $M>0$ be such that
 \begin{equation}\label{E:G.2.1}
 y^* H_{pp}(\vrt)\, y \ge \frac 1M\, \lV y \rV^2
 \qquad \text{for all}\quad \vrt\in \La.
 \end{equation}
 Let $R>0$ be such that
 \begin{equation}\label{E:G.2.2}
 \lv \, y^* D(\vrt)\, y \,\rv < M\, R^2
 \qquad \text{for all}\quad \vrt\in \La,
 \quad \lV y\rV=1.
 \end{equation}
 We claim that
 $$
 \lv\,y^* V(t)\, y \,\rv \le M\, R\, \coth(R)
 \qquad
 \text{for all}\quad \lV y \rV = 1,
 \quad
 1<t<T-1.
 $$
 
  For, let  $\lV x \rV=1$ and write 
 $H_{pp}(t):= H_{pp}(\psi_t\,\theta)$.
 Suppose that there exists $1<t_0<T$ such that
 $$
 x^* V(t_0)\, x =: M\,\a > M\, R\, \coth(R).
 $$
 There exists $d_0\in\re$ such that $R\, \coth\, (R\, t_0 - d_0) =\a$.
 Observe that $d_0>0$ and $t_0> \frac {d_0}R>0$. Write
 $w(t):= R\, \coth\, (R\, t - d_0)$. Then $w(t)$ is a
 solution of $\dot{w}+w^2-R^2=0$ for $t>\frac{d_0}R$.
 In particular,
 $$
 M\, \dot{w} + M\, w^2 - M\, R^2 = 0.
 $$
 
 Let $f(t):= x^*\, V(t)\, x - M\, w(t)$. Then $f(t_0)=0$ and
 \begin{equation}\label{E:G.2.3}
 f'(t) + \left(x^*\, V(t)^* H_{pp}(t)\, V(t)\, x - M\, w(t)^2\right)
 + \left( x^* D(t)\, x + M\, R^2\right) = 0\, .
 \end{equation}
 Using Schwartz inequality and~\eqref{E:G.2.1}, we have that
 \begin{align*}
 M\, w(t_0)^2 &= M\, \a^2 = \frac 1M \, (M\, \a)^2
                 = \frac 1M\, ( x^*\, V(t_0)\, x)^2
 \\
              &\le \frac 1M\, \lv V(t_0)\, x\rv^2
                 \le \bigl(V(t_0)\,x\bigr)^* 
                 H_{pp}(t_0)\, \bigl( V(t_0)\,x\bigr),
 \end{align*}
 \begin{equation}\label{E:G.2.4}
 M\, w(t_0)^2 \le x^*\, V(t_0)^* H_{pp}(t_0)\, V(t_0)\, x.
 \end{equation}
 Then~\eqref{E:G.2.4}, \eqref{E:G.2.2} and~\eqref{E:G.2.3} imply that
 $f'(t_0)<0$. The same argument can be applied each time that $f(t)=0$.
 Therefore $x^*\, V(t)\, x \le M\, w(t)\text{ for all }  t_0<t<T$ and
 $$
 x^*\, V(t)\, x \ge M\, w(t)\quad\text{for all } \tfrac{d_0}R<t< t_0.
 $$
 Then
 $$
 \lim_{t\to\left(\frac{d_0}R\right)^+} x^* \, V(t)\, x
 \ge
 \lim_{t\to\left(\frac{d_0}R\right)^+} M\, w(t)
 = +\infty\, .
 $$
 Since $\frac{d_0}R>0$, this contradicts the existence of $V(t)$ for $t>0$.
 Hence such $t_0$ does not exist and
 $$
 x^* \, V(t)\, x \le M \, R\, \coth(R)\qquad \text{ for all } 1<t< T.
 $$

 Now suppose that there exist $t_1\in ]0,T-1[$ and $\lV z\rV=1$ such that
 $$
 z^*\, V(t_1)\, z < - M\, R\,\coth(R)=M\,R\,\coth(-R) .
 $$
 We compare $v(t)= z^*\, V(t)\, z$ with $M\, w_1(t)$, where
 $w_1(t)= R\, \coth\, (R\, t- c_0)$ is such that $M\, w_1(t_1)=v(t_1)$.
 Observe that $w_1(t)$ is defined for $t<\frac{c_0}{R}$ and that in
 this case $-R<R \,t_1-c_0<0$. Therefore
 $0<t_1<\frac{c_0}R$ and $\frac{c_0}R<t_1+1<T$.
  The same argument
 as above shows that
 $v(t)\le M\, w_1(t)$ for $t_1\le t<\frac{c_0}R$. Since
 $\lim_{t\to\left(\frac{c_0}R\right)^-}w_1(t)=-\infty$,
 this contradicts the fact that $V(t)$ is defined 
 on $[t_1,t_1+1]\supset[t_1,\frac{c_0}R]$.
 
 \qed

 \begin{Proposition}\label{growth}\quad
 
 Let $\La\subset T^*M\times\re$ and
 $H\in\cH(T^*M\times\re,\La,b_1,b_2)$ be a bounded
 hamiltonian. Let $\theta\in\La$ and suppose that
 the orbit $\Psi_{\re}(\theta)$ is inside $\La$ and
 is disconjugate.  
 \begin{enumerate}
 \item\label{g1} For $t\ge 0$ let 
 $Y_\theta(t):\V(\theta)\to \H(\psi_t(\theta))$ be
 defined by $Y_\theta(t)\,u:=d\pi(d\psi_t(0,u))$.
  Then for all $R>0$ there is $T=T(\theta,R,H)>0$ such that
 $$
 | Y_\theta(t)\,v|>R\,|v|\quad \text{ for all }\quad
 t>T \text{ and } v\in\V(\theta).
 $$
 \item\label{g2} Let $\B(\th):=\big\{\,\xi\in T_\th T^*M\;|\;
 \sup_{t\in\re}|d_\th\psi_t(\xi)|<+\infty\;\big\}$. Then
 $$
 \B(\th)\subset\E(\th)\cap\F(\th).
 $$
 \end{enumerate}
 
 If $H$ is autonomous, $X=(H_p,-H_x)$ is the hamiltonian vector field,
 $\Si=H^{-1}\{k\}$ and $k=H(\th)$ is a regular value of $H$, then
 \begin{enumerate}
 \setcounter{enumi}{2}
 \item\label{g3} $X(\th)\in\E(\th)\cap\F(\th)\subset T_\th\Si$.
 \item\label{g4} $H_p(\th)=d\pi(X(\th))\ne 0$.
 \end{enumerate}
 \end{Proposition}

 \begin{Lemma}\label{Jform}\quad
 
 Let $(H_1,V_1)$ be a matrix solution of the Jacobi
 equation~\eqref{Jacobi} such that $\det H_1\ne 0$ for all $s\in[t_0,t]$
 and $\image(H_1,V_1)$ is a lagrangian subspace.
 
 Then any other matrix solution of~\eqref{Jacobi} is of the form
 \begin{align}
 H_2 &=H_1\left[D+\int_{t_0}^t H_1^{-1}\,H_{pp}\,(H_1^*)^{-1}\,K \;
 ds\right], \label{H2}\\
 V_2 &= (H_1^*)^{-1}\big[ K+ V_1^* H_2\big],
 \label{V2}
 \end{align}
 where $K$ and $D$ are constant matrices.
 
 Moreover $\image(H_2,V_2)$ is a lagrangian subspace iff $K^*D$ is
 symmetric.
 
 Conversely, for any $K$, $D$, formulas~\eqref{H2}, \eqref{V2} define solutions
 of~\eqref{Jacobi} on the largest interval on which $\det H_1\ne 0.$
 \end{Lemma}
 
 \begin{proof}
 If $(H_i,V_i)$, $i=1,2$ are matrix solutions of the Jacobi equation,
 then
 \begin{equation}\label{wronsk}
 K:= H_1^*\,V_2-V_1^* H_2
 \end{equation}
 is a constant matrix. This can be seen directly by
 differentiating~\eqref{wronsk} and using~\eqref{Jacobi}, and
 corresponds to the preservation of the symplectic form.
 From~\eqref{wronsk} we get~\eqref{V2}.
 
 If $\rank(H,V)=n$, the subspace $\image(H,V)$ is lagrangian iff
 $V^*H$ is symmetric, or equivalently $S=V H^{-1}$ is symmetric (when
 $\det H\ne 0$).
 
 Let $(H_2,V_2)$ be a matrix solution of~\eqref{Jacobi}. Write
 \begin{equation}\label{Z2}
 H_2 = H_1\, Z.
 \end{equation}
 Using~\eqref{Jacobi} and~\eqref{V2} we get
 \begin{align}
 \dot H_2 &= \dot H_1 Z+ H_1 \dot Z
          = H_{px}\,H_1 Z+ H_{pp}\, V_1 Z+ H_1 \dot Z,
          \label{h21}\\
 \dot H_2 &= H_{px} H_2+ H_{pp} V_2
        \notag \\
        &= H_{px}\, H_1 Z 
       + H_{pp}\, (H_1^*)^{-1}\,\big[ K+ V_1^* H_1 Z\big].	  
       \label{h22}
 \end{align}
 Equating \eqref{h21} to \eqref{h22} and using that $V_1^* H_1= H_1^* V_1$ we obtain
 \begin{equation}\label{dZ2}
 \dot Z = H_1^{-1}\, H_{pp}\, (H_1^*)^{-1}\, K.
 \end{equation}
 From~\eqref{Z2} and ~\eqref{dZ2} we get~\eqref{H2}.
 
 Using~\eqref{V2}, \eqref{H2} and that $K$ is constant, we have that 
 \begin{align*}
 V_2^* H_2 &=\big[ H_2^* V_1+ K^*\big] H_1^{-1} H_2 \\
 &= H_2^*\,(V_1 H_1^{-1})\,H_2 + K^* H_1^{-1} H_2 \\
 &= H_2^* \,(V_1 H_1^{-1})\, H_2 + K^* D + \int_{t_0}^t K^* H_1^{-1} \,
 H_{pp}\,(H^*_1)^{-1} K\; ds. 
 \end{align*}
 Since $V_1 H_1^{-1}=S$ is symmetric, $V_2^* H_2$ is symmetric iff
 $K^* D$ is symmetric.
 
   Direct calculation shows that~\eqref{H2},~\eqref{V2} is a solution
   of~\eqref{Jacobi}.
   
 \end{proof}

 \noindent{\bf Proof of Proposition~\ref{growth}:}\quad
 
 \eqref{g1}. Let $(H,V)$ be the solution of the Jacobi
 equation~\eqref{Jacobi} with $H(-1)=0$, $H(0)=I$. Then
 $\image(H,V)=d\psi_t(\V(\psi_{-1}(\th)))$ is lagrangian and $\det
 H(t)\ne 0$ if $t\ne -1$. Write $Y(t)=Y_\th(t)$ and let $(Y,U)$ be the
 solution of~\eqref{Jacobi} with $Y(0)=0$, $U(0)=I$. Using
 lemma~\eqref{Jform} and \eqref{wronsk} with $K=Y^* V-U^* H\equiv -I$, $D=Y(1)^{-1}
 H(1)$, $t_0=1$, we get that
 $$
 H=Y\,\left[ D-\int_1^t Y^{-1}\,H_{pp}\,(Y^*)^{-1}\;ds\right]
 \qquad\text{ for }\quad t>0,
 $$ 
 because in~\eqref{H2}, \eqref{V2}, $H_2(1)=H(1)$, $V_2(1)=V(1)$.
 
 From lemma~\ref{Jform} with $t_0=0$ (here $K$ changes sign) we also
 get
 $$
 Y = H \left[\int_0^t H^{-1}\,H_{pp}\,(H^*)^{-1}\;ds\right]
 \qquad\text{ for }\quad t>0.
 $$
 Therefore 
 $$
 \left[\int_0^t H^{-1}\,H_{pp}\,(H^*)^{-1}\;ds\right]
 \left[ D-\int_1^t Y^{-1}\,H_{pp}\,(Y^*)^{-1}\;ds\right]
 =I\qquad\text{ for all }\quad t>0.
 $$
 The first factor is positive definite. Then the second factor, being
 the inverse of the first factor, is also positive definite. The
 family of matrices $t\mapsto\int_1^t Y^{-1}\,H_{pp}\,(Y^*)^{-1} ds$
 is increasing in the order $\prec$ and is bounded above by $D$.
 Then the limit
 $$
 M(a):=\int_a^\infty Y^{-1}\,H_{pp}\,(Y^*)^{-1}\;ds,\qquad a>0
 $$
 exists and is finite.
 
 Let 
 \begin{gather*}
 Z(t)=Y(t)\,M(t),\\
 W=(Y^*)^{-1}\,\big[-I+U^* Z\,\big].
 \end{gather*}
 Then $(Z,W)$ is the solution of~\eqref{Jacobi} with $\image(Z,W)$
 lagrangian given by lemma~\ref{Jform} for $t_0=+\infty$, $K=-I$,
 $D=0$. This can be seen directly or by taking the limit of the
 solutions\footnote{In fact $(Z,W)$ is a solution of the Jacobi
 equation corresponding to the ``stable'' Green bundle 
 \linebreak
 $\E=\image(Z,W)$.}
  $(Z_c,W_c)$ with $t_0=c$, $K=-I$, $D=0$ when $c\to+\infty$.

 Writing $S=U Y^{-1}$ and $\bbS=W Z^{-1}$ we get 
 $$
 (Y^*)^{-1} M^{-1} Y^{-1}=(Y^*)^{-1} Z^{-1}
 =-W Z^{-1}+(Y^*)^{-1} U^* = -\bbS+S^*,\qquad\text{ for } t>0.
 $$
 By proposition~\ref{BdRic}, $\lV S(t)\rV<A$ and $\lV\bbS(t)\rV<A$ for
 all $t>1$.
  Let  $\la(t)$ be the largest eigenvalue of $M(t)$, $t>0$. Since
  $M(t)$ is positive definite $\la(t)>0$ and $\lV M(t)\rV=\la(t)$.
  Moreover, if $|x|=1$,
  \begin{gather*}
  2A\ge x^*\, Y^*(t)^{-1} \,M(t)^{-1} \,Y(t)^{-1}\, x
  \ge \tfrac 1{\la(t)}\;|Y(t)^{-1}x|^2,\\
  |Y(t)^{-1}x|\le\sqrt{2A}\,\lV M(t)\rV^{\frac 12}
  \qquad\text{ if }\quad |x|=1,\quad t>0.
  \end{gather*}
  Then if $|v|=1$,
  $$
  |Y(t)\,v|\ge\frac 1{\lV Y(t)^{-1}\rV}
  \ge\frac 1{\sqrt{2A}\,\lV M(t)\rV^{\frac 12}}
  \qquad \text{ if }\quad t>0,\quad |v|=1.
  $$
  Since $M(t)\to 0$ when $t\to+\infty$, this implies item~\eqref{g1}.

  \eqref{g2}.  Suppose that $\xi\in\B(\th)$. Let $\zeta_T\in
  E_T:=d\psi_{-T}(\V(\psi_T(\th)))$ be such that
  $d\pi(\zeta_T)=d\pi(\xi)$. Then $(0,v_T):=\xi-\zeta_T\in\V(\th)$. We
  have that $Y_T(\th)\; v_T =d\pi(d_\th\psi_T(\xi))$ is bounded on
  $T$. By item~\eqref{g1}, $\lim_{T\to+\infty}v_T=0$. Thus
  $\lim_{T\to+\infty}\zeta_T=\xi$. Since $\lim_{T\to+\infty}E_T=\E$, 
  we have that
  $\xi=\lim_{T\to+\infty}\zeta_T\in\E$.
  
  Similarly, if $\eta_T\in d\psi_T(\V(\psi_{-T}(\th)))$ is such that
  $d\pi(\eta_T)=d\pi(\xi)$ then $\xi=\lim_{T\to+\infty}\eta_T\in\F$. 

  \eqref{g3}. From item~\eqref{g2},
  $X(\th)\in\B(\th)\subset\E(\th)\cap\F(\th)$. From the hamiltonian
  equation $\om(X,\,\cdot\,)=-dH$, we have that
  $T_\th\Si=\{\,\xi\;|\;\om(X,\xi)=0\,\}$. Since $\E(\th)$ is
  lagrangian and $X(\th)\in\E(\th)$, we have that $\E(\th)\subset
  T_\th\Si$. Similarly, $\F(\th)\subset T_\th\Si$.
  
  \eqref{g4}. If $H_p(\th)=0$ then $X(\th)\in\V(\th)\cap \E(\th)=\{0\}$.
  This implies that $dH(\th)=0$,
   which contradicts the regularity of $k$.
    \qed

\section{The index form}

Let $\th\in T^*M$ and $E\subset T_\th T^*M$ be a lagrangian subspace.
Write $E(t):=d\psi_t(E)$. Suppose that $E(t)\cap\V(\psi_t(\th))=\0$
for all $t\in]a,b[$. At the endpoints of the interval $]a,b[$ the
subspace $E(t)$ may intersect the vertical $\V$. Let $H(t)$, $V(t)$ be
a matrix solution of the Jacobi equations~\eqref{Jacobi} such that
$E(t)=\image(H(t),V(t))$, with $\det H(t)\ne 0$ for all $t\in]a,b[$.
Since $E(t)$ is lagrangian, the corresponding solution of the Riccati
equation $S(t)=V(t)H(t)^{-1}$ is symmetric. In particular
\begin{equation}\label{sym}
H(t)^*V(t)=V(t)^*H(t).
\end{equation}

Let
$$
\cH^1_\th[a,b]:=\Big\{\,\xi:[a,b]\to TM\,\Big|\,
\pi\circ\xi(t)=\pi\circ\psi_t(\th),\;\int_a^b|\xi'(s)|^2\,ds<+\infty\,\Big\}.
$$
Given two vector fields $\xi,\,\eta\in\cH^1_\th[a,b]$ along
$t\mapsto\pi(\psi_t(\th))$ define $\zeta,\,\rho\in\cH^1_\th[a,b]$ by
\linebreak
$\xi(t)=H(t)\,\zeta(t)$ and $\eta(t)=H(t)\,\rho(t)$. Recall that the
index form is
\begin{equation}\label{index2}
I(\xi,\eta)=\int_a^b \big[\,\dxi\, L_{vv}\,\deta +\dxi\, L_{vx}\,\eta 
  +\xi\, L_{xv}\,\deta  +\xi L_{xx}\eta\,\big] \;dt.
\end{equation}

\begin{Proposition}\label{Iform}
If the hamiltonian $H$ is convex and $C^2$ then
 \begin{align}
 I(\xi,\eta)
 &=\int_a^b (H\zeta')^*\, (H_{pp})^{-1} (H\rho')\;dt
 + (H \zeta)^* (V\rho)\big\vert_a^b
 \notag
 \\
 &=\int_a^b (H\zeta')^*\, (H_{pp})^{-1} (H\rho')\;dt
 + \xi\, S\,\eta\big\vert_a^b.
 \label{info0}
 \end{align}
 
  If $a=t_0<t_1<\cdots<t_N=b$ is a partition of $[a,b]$
 and for the integral we use different non-vertical
 lagrangian subspaces $E_i(t)=d\psi_{t-t_i}(E_i)$
 on each subinterval $]t_i,t_{i+1}[$ the formula becomes
 \begin{equation}\label{info}
  I(\xi,\eta)
 =\sum_{i=1}^N\int_{t_{i-1}}^{t_i} 
 (H_i\zeta_i')^*\, (H_{pp})^{-1} (H_i\rho_i')\;dt 
 + \sum_{i=1}^N(H_i \zeta_i)^* (V_i\rho_i)\Big\vert_{t_{i-1}}^{t_i},
 \end{equation}
 where $\xi = H_i \zeta_i$ and $\eta = H_i \rho_i$ on
 $[t_{i-1},t_i]$.
\end{Proposition}

 \begin{proof}
 Let $\cF(x,v)=(x,L_v(x,v))$ be the Legendre transform. We have that
 $$
 D\cF\begin{bmatrix} H\\  H'\end{bmatrix}
 =\begin{bmatrix} 
   I      &   0 \\
   L_{vx} &  L_{vv}
   \end{bmatrix}
   \begin{bmatrix} H \\ H' \end{bmatrix}
  =\begin{bmatrix} H \\ V \end{bmatrix} .
 $$
 Therefore
 \begin{equation}\label{VHd}
 V = L_{vx} \,H + L_{vv}\, H'.
 \end{equation}
 Since the maximum in~\eqref{hami} is obtained in $L_v=p$ and also
 $v=H_p$,
 $$
 H(x,L_v(x,v))= v\cdot L_v(x,v)-L(x,v).
 $$
 Differentiating with respect to $x$ we get
 $$
 H_x(x,L_v(x,v))=-L_x(x,v).
 $$
 Also
 \begin{equation}\label{Lxx}
 H_{xx} + H_{xp}\; L_{vx} = - L_{xx}.
 \end{equation}
 From the Jacobi equation~\eqref{Jacobi},
 \begin{equation}\label{Vjac}
 V'= -H_{xx}\;H-H_{xp}\; V.
 \end{equation}
 Replacing~\eqref{VHd} and~\eqref{Lxx} in~\eqref{Vjac} we get
 \begin{equation}\label{VLxx}
 V'=L_{xx}\;H- H_{xp}\;L_{vv}\;H'.
 \end{equation}
 Since the inverse of the Legendre transform is
 $\cF^{-1}(x,p)=(x,H_p(x,p))$, we have that
 $$
   \begin{bmatrix} 
   I      &   0 \\
   H_{px} &  H_{pp}
   \end{bmatrix}
   \begin{bmatrix} 
   I      &   0 \\
   L_{vx} &  L_{vv}
   \end{bmatrix}
   =
   \begin{bmatrix} 
   I      &   0 \\
   0      &   I
   \end{bmatrix}.
 $$
 Therefore 
 \begin{gather*}
 H_{px}+H_{pp}\; L_{vx} = 0
 \qquad\text{ and } \qquad
 L_{vv}= H_{pp}^{-1}.
 \end{gather*}
 Replacing $H_{xp}=H_{px}^*= -L_{xv}\;H_{pp}$ in~\eqref{VLxx} we get
 \begin{equation}\label{Vd}
 V' = L_{xx} H + L_{xv} H'.
 \end{equation}
 
 Write $\xi(t)= H(t)\,\zeta(t)$. Using~\eqref{VHd} and~\eqref{Vd}, the
 integrand for $I(\xi,\xi)$ 
 in~\eqref{index2} 
 can be written as
 \begin{equation*}
 (L_{vv}\,H\,\zeta'+V\,\zeta)\cdot(H'\, \zeta+H\,\zeta'\,)
 +(V'\,\zeta+L_{xv}\,H\,\zeta'\,)\cdot H\,\zeta\,.
 \end{equation*}
 Since $(L_{vv}\,H\,\zeta')\cdot(H' \zeta)
 =(H\zeta')\cdot ( L_{vv}\, H' \zeta)$ and 
 $(L_{xv}\, H \zeta')\cdot(H\zeta)=(H \zeta')\cdot(L_{vx}\, H\zeta)$,
 using~\eqref{VHd} 
 the integrand can be written as
 $$
 L_{vv}\, H \zeta'\cdot H \zeta'+H \zeta'\cdot V\zeta
 +V\zeta\cdot H'\zeta+V\zeta\cdot H\zeta'+V'\zeta\cdot H\zeta.
 $$
 Using~\eqref{sym} we have that $H\zeta'\cdot
 V\zeta=V^*H\zeta'\cdot\zeta=H^* V\zeta'\cdot \zeta
 =V\zeta'\cdot H\zeta$. Hence the integrand in~\eqref{index2} is
 $L_{vv}\, H\zeta'\cdot H\zeta'+(V\zeta\cdot H\zeta)'$.
 Therefore
 $$
 I(\xi,\xi)=\int_a^b (L_{vv}\, H\zeta'\cdot H\zeta')\;dt
 +H\zeta\cdot V\zeta\big|_0^T
 \qquad\text{ if } \quad\xi=H\,\zeta\in\cH^1_\th[a,b].
 $$
 Since $L_{vv}=H_{pp}^{-1}$ this formula implies the proposition.
 
 \end{proof}
 
 For $\th\in T^*M$ and $T>0$ write
 $$
 \Om(\th,T):=\big\{\;\xi:[0,T]\to TM\;\big|\;\xi\in\cH^1_\th[0,t],\;
 \xi(0)=\xi(T)=0\;\big\}.
 $$

 \begin{Lemma}\label{I+}
 The orbit segment $\psi_{[0,T]}(\th)$ is disconjugate if and
 only if the index form $I$ on $\Om(\th,T)$ is positive definite.
 \end{Lemma}

 \begin{proof}
 If $\psi_{[0,T]}(\th)$ has a conjugate point then there are
 $0<a<b<T$ and $v\in T_{\psi_a(\th)}M$, $|v|=1$ such that
 $d\pi\,d_{\th_a}\psi_{b-a}(0,v)=0$, where $\th_a=\psi_a(\th)$.
 Let $\xi(t)=0$ if $t\notin[a,b]$, and
 $\xi(t)=d\pi\,d_{\th_a}\psi_{t-a}(0,v)$ if $t\in[a,b]$. Using
 $H(t)=d\pi\circ d_{\th_a}\psi_{t-a}|_{\V(\th_a)}=Y_{\th_a}(t-a)$ in
 proposition~\ref{Iform}, we have that $I(\xi,\xi)=0$.
 
 If $\psi_{[0,T]}(\th)$ is disconjugate then there is $\de>0$ such
 that $\psi_{[-\de,T+\de]}$ is disconjugate (see for
 example~\cite[prop.~1.15]{CI}). Let
 $H(t)=d\pi\circ d\psi_{t+\de}|_{\V(\th_{-\de})}=Y_{\th_{-\de}}(t+\de)$.
 Let $\xi\in\Om(\th,T)$ with $\xi(t)\not\equiv 0$.
 Since $\xi(0)=0$, $\xi(T)=0$ and $\det H(t)\ne 0$ for all $t\in[0,T]$, 
 if we write $\xi= H\,\zeta$ then $\zeta'\not\equiv 0$.
 Using proposition~\ref{Iform}, $I(\xi,\xi)>0$.
 
 \end{proof}

\bigskip

   \section{The autonomous case.}\label{Saut}
   
   A time dependent hamiltonian preserves a natural smooth foliation
   transversal to the flow given by fixing the time coordinate.
   But an autonomous hamiltonian on an energy level $\Si$ may 
   not have a continuous subbundle of $T\Si$ which is invariant
   under $d\psi_t$ and transversal to the flow direction. The
   horocycle flow is an example of this phenomenon, 
   see ~\cite[example~A.1]{CI}. In this section we show that there is
   a reparametrization of the flow on a neighbourhood of any orbit 
   which is the flow of a time dependent hamiltonian.
   This will allow us to get hyperbolicity.
   
   \subsection{The transversal behaviour.}\quad

   Let $\Si$ be a compact and regular energy level, 
   $\th_0\in\Si$ and suppose that
   the orbit $\psi_\re(\th_0)$ is disconjugate. Let
   $\th\in\psi_\re(\th_0)$. By~\ref{growth}.\eqref{g4},
   $H_p(\th)=d\pi(\X(\th))\ne 0$. Let
   $$
   \N(\th):=\{\,\xi\in T_\th\Si\;|\;
   \langle d\pi(\xi),d\pi (X(\th))\rangle_{\pi(\th)}=0\,\}.
   $$
   Then $\N(\th)\oplus\langle X(\th)\rangle=T_\th\Si$. Let
   $(q_1,\ldots,q_n,t)$ be a smooth coordinate system on $M\times\re$
   along\footnote{The coordinate system is taken in $M\times\re$ in order 
   to deal with possible self intersections of the projection $\pi(\psi_\re(\th_0))$ of the orbit to $M$.}
   the projection $(\pi(\psi_t(\th_0)),t)$ of the orbit of
   $\th_0$ such that 
   \begin{enumerate}
   \item $\partial/\partial  q_1|_{(\pi(\psi_t(\th_0)),t)}
   =d\pi\,X(\psi_t(\th_0))$.
   \item $\partial/\partial q_2,\ldots,\partial/\partial q_n$ is an
   orthonormal basis for
   $d\pi(\N(\psi_t(\th_0)))=\langle\partial/\partial q_1\rangle^\perp$
   along $(\pi(\psi_t\,\th_0),t)$.
   \end{enumerate}

    Write $p_i=dq_i|_{T_qM}$, $i=1,2,\ldots,n$. We have that
    $$
    1=\tfrac{d}{dt}q_1|_{(\pi(\psi_t \th_0),t)}
    =H_{p_1}(\psi_t(\th_0))\ne 0.
    $$
    Then the equation $H(q_1,q_2,\ldots,q_n;p_1,p_2,\ldots,p_n)=k$
    can be solved locally for $p_1$:
    $$
    p_1= - K(T,Q,P),
    $$ 
    where $P=(p_2,\ldots,p_n)$, $Q=(q_2,\ldots,q_n)$, $T=q_1$.
    This solution can be extended to a simply connected neighbourhood
    $W$ of the orbit $\{(\psi_t(\th_0),t)\;|\;t\in\re\}\subset
    T^*M\times\re$. 

    Let $\phi_T$ be the reparametrization of the flow $\psi_t$ on $W$
    such that it preserves the foliation
    $q_1=T=\text{constant}$. In particular $d_{\th_0}\phi_T$ preserves
    the transversal bundle $\N(\phi_T(\th_0))$ along the orbit of
    $\th_0$. Define $\tau(T,\th)$ by
    $\phi_T(\th)=\psi_{\tau(T,\th)}(\th)$. Then $\tau(T,\th_0)=T$ for
    all $T\in\re$.
    
    In~\cite[section~2]{CI} the following statements are proven:
    \begin{enumerate}
    \item\label{a1} For $\th\in W$,
    $$
    \phi_T(\th)=(T=q_1,Q(T),p_1(T),P(T);\tau(T,\th))\in W,
    $$ 
    where $T\mapsto(T,Q(T),P(T))$ are the orbits of the hamiltonian
    $K(T,Q,P)$  defined by $H(T,Q,-K,P)\equiv k$.
    
    In fact the hamiltonian vector field for $K$ is
    $$
    Y=(K_P,-K_Q) = \la\;(H_P,-H_Q),  \qquad \la=H_{p_1}^{-1}.
    $$

    \item\label{a2} If $\P(\th):T_\th\Si\to\N(\th)$ 
    is the projection along the
    flow direction $\P(\xi+a\,X)=\xi\in\N(\th)$, then
    $$
    d\phi_T(\xi)=\P\big(d\psi_{\tau(T,\th)}(\xi)\big),
    \qquad\forall \xi\in\N(\th).
    $$
    \item\label{a3} Let 
    $$
    \H=\begin{bmatrix}
     H_{pq} & H_{pP}\\
     -H_{Qq} &-H_{QP}
     \end{bmatrix}_{(2n-1)\times(2n-1)},
     \quad
     \K=\begin{bmatrix}
     \b0_{1\times n} & \b0_{1\times(n-1)}\\
     K_{Pq} & K_{PP} \\
     -K_{Qq} & -K_{QP}
     \end{bmatrix}_{(2n-1)\times(2n-1)}.
    $$
    be the matrices corresponding to the Jacobi equations for the
    hamiltonians $H$ and $K$ in the coordinates $(T=q_1,Q,P)$.
    If $\th_T:=\psi_T(\th_0)$, then
    \begin{enumerate}
    \item\label{Kfor} If $\xi\in\N(\th_T)$ then
          $\K(\th_T)\,\xi=\P(\th_T)\,\H(\th_T)\,\xi$.
    \item\label{Konv} If $(0,V)\in\N(\th_T)\setminus\0$, then
          $$
	  V^*\,K_{PP}(\th_T)\,V=V^*\,H_{PP}(\th_T)\, V>0.
	  $$ 	  
    \item\label{Kbd} The operator $\K(\th_T)|_{\N(\th_T)}$ is uniformly bounded
    on $T\in\re$.	  
    \end{enumerate}
    \end{enumerate}
    
    We obtain that $K$ is a bounded hamiltonian. Indeed, the constant $b_2$ is the
    same as the constant for $H$ along the orbit of $\th_0$ by
    item~\eqref{Konv}. Items~\eqref{a1}, \eqref{Kbd} and $\max_\Si|p_1|$
    give a bound for the $C^3$ norm of $K$. In fact item~\eqref{Kbd}
    follows from item~\eqref{Kfor} because the angle $\varangle(X,\N)$
    is bounded away from zero by the compacity of $\Si$. Therefore the
    bound $b_1$ can also be taken uniform for all $\th_0\in\Si$.
     
    In particular we can apply the results of the previous sections to
    the hamiltonian $K$. The following statements are also proved
    in~\cite[corollary~2.3]{CI}:
    \begin{enumerate}
    \setcounter{enumi}{3}
    \item If the orbit $\psi_\re(\th_0)$ is disconjugate then 
    the orbit of $\th_0$ under $\phi_T$ is disconjugate.
    
    \item The Green bundles for $d\psi_s|_{\N(\th_T)}$: 
    \begin{align*}
    \E^\top(\th_T)&=\lim_{s\to+\infty}d\psi_{-s}
    \big(\V(\th_{T+s})\cap\N(\th_{T+s})\big)\\
    \F^\top(\th_T)&=\lim_{s\to+\infty}d\psi_{s}
    \big(\V(\th_{T-s})\cap\N(\th_{T-s})\big)
    \end{align*}
    satisfy: $\E^\top=\E\cap\N$ and $\F^\top=\F\cap\N$.
    \end{enumerate}

    \subsection{The graph transform.}\quad
    
    Let $\Psi_t=\P\circ d\psi_t$ be action on $\N$ induced by
   $d\psi_t$. Along each disconjugate orbit $\th\in\La$,
   $\Psi_t=d_\th\phi_t|_\N$, where $\phi_t$ is the reparametrization
   of $\psi_t$ described above.

   \begin{Proposition}\label{Xhyp}
   If $\La$ is compact and $\Psi_t|_\La$ is hyperbolic, then
   $d\psi_t|_\La$ is hyperbolic. 
   \end{Proposition}  
   
   \begin{proof}
   Let $\N=E^s\oplus E^u$ be the hyperbolic splitting for $\Psi_t$.
   We look for the hyperbolic splitting
   $T_\th\Si=\cE^s\oplus\X\oplus\cE^u$ for $d\psi_t$ on $\La$.
   We only construct $\cE^u$, the construction of $\cE^s$ is similar.
   
   Let 
   $$
   \cF=\{\,L:E^u\to\re\text{ continuous }|\;\forall\th\in\La\quad L_\th:E^u(\th)\to\re
   \text{ is linear }\}
   $$
   with $\lV L\rV=\sup_{\th\in\La}\lV L_\th\rV$.
   For $L\in\cF$ we associate a subbundle $W_L\subset T\Si$ by
   $$
   W_L(\th)=\graph(L)=\{\,v+L_\th(v)\,X(\th)\;|\;v\in E^u(\th),
   \;\th\in\La\;\}.
   $$
   Define the graph transform $T_t:\cF\to\cF$, $t\in\re$ by
   $W_{T_t L}=d\psi_t(W_L)$, i.e.
   $$
   d\psi_t(v+L(v)\,X)=\Psi_t(v)+(T_tL)(\Psi_t \,v)\cdot X,\qquad
   \text{ if }v\in E^u\subset\N.
   $$
   We show that $T_t$ is a contraction for $t>0$ sufficiently large.
   Indeed,
   \begin{gather*}
   T_t(L_1-L_2)\circ\Psi_t|_\N=d\psi_t[(L_1-L_2)\cdot X]=(L_1-L_2)\cdot
   (X\circ\psi_t),\\
   T_t(L_1-L_2)=\big[(L_1-L_2)\circ \Psi_t^{-1}|_{E^u}\big]\;X,
   \\
   \lV T_t\rV\le \lV\Psi_t^{-1}|_{E^u}\rV\;\lV X\rV<1
   \qquad\text{ for $t>0$ large.}
   \end{gather*}
       Thus $T_t$  has a fixed point
   $L^*$ and  $\cE^u:=\graph(L^*)$ is $d\psi_t$-invariant.

   If $\xi\in\cE^u$ we have that $\P(\xi)\subset E^u$ and
   $$
   d\psi_{-t}(\xi)=\Psi_{-t}\circ\P(\xi) + L^*(\Psi_{-t}\circ\P(\xi)),
   \qquad \text{ if }\xi\in\cE^u.
   $$
   Therefore
   $$
   \lV d\psi_{-t}\vert_{\cE^u}\rV\le\lV\Psi_{-t}|_{E^u}\rV\,\lV\P\rV\,
   \big[1+\lV L^*\rV\big].
   $$
   Since $\lV\Psi_{-t}|_{E^u}\rV\to 0$ exponentially when
   $t\to+\infty$ and $\lV\P|_\La\rV$, $\lV L^*\rV$ are
   bounded\footnote{The norm $\lV\P|_\La\rV$ is bounded because
   $\varangle(X,\N)>0$ is bounded below on the compact set $\La$.}
   we get that $\lV\Psi_{-t}|_{E^u}\rV\to 0$ exponentially when
   $t\to+\infty$.
 
   \end{proof}
 
 \section{Hyperbolicity.}

  Let $\pi:\bE\to\bB$ be a continuous vector bundle over a compact
  metric space $\bB$. Let $\Psi:\re\to\isom(\bE)$ be a continuous
  $\re$-action of bundle maps, i.e.
  \begin{enumerate}
  \item There is a continuous flow $\psi_t$ on $\bB$ such that 
  $\pi\circ\Psi_t=\psi_t\circ\pi$ for all $t\in\re$.
  \item $\Psi_t:\bE(b)\to\bE(\psi_t(b))$ is a linear isomorphism
   for all $(b,t)\in\bB\times\re$.
   \item $\Psi_s\circ\Psi_t=\Psi_{s+t}$ for all $s,t\in\re$.
  \end{enumerate}
  We say that $\Psi$ is {\it hyperbolic} if there is a
  $\Psi_t$-invariant splitting
  $\E=E^s\oplus E^u$ and $T>0$ such that
  $$
  \lV \Psi_T|_{E^s}\rV<\tfrac 12,\qquad
  \lV \Psi_{-T}|_{E^u}\rV<\tfrac 12.
  $$
  This condition implies that the subbundles $E^s$ and $E^u$ are
  continuous.
  
  We say that $\Psi$ is {\it quasi-hyperbolic} if for
  all $\xi\in\bE$ with $\xi\ne 0$ we have that
  $$
  \sup_{t\in\re}|\Psi_t(\xi)|=+\infty.
  $$
  A quasi-hyperbolic action is hyperbolic provided that the projected
  flow $\psi_t|_\bB$ is chain-recurrent or if in \eqref{Es}
  $\dim E^s$ is constant on all the minimal sets
  (see Sacker \& Sell~\cite{SaSe},
  Churchill \& Selgrade~\cite{Churchill-Selgrade}, Ma\~n\'e~\cite{Ma10}).
  See also Appendix~\ref{apa}.
  Examples of quasi-hyperbolic dynamical systems which are not
  hyperbolic appear in Robinson~\cite{Ro4} and Franks \&
  Robinson~\cite{FraRo}.
  Here we will drop any recurrency hypothesis in the case of
  hamiltonians without conjugate points.
  This was also done in   Johnson, Novo, Obaya \cite{JNO}
  and in Arnaud~\cite{Arnaud4}.
    
  For $b\in B$ define
  \begin{align}
   E^s(b)&:=\big\{\,\xi\in\bE(b)\;\big|\;
   \sup_{t>0}|\Psi_t(\xi)|<\infty\,\big\},
   \label{Es}\\
   E^u(b)&:=\big\{\,\xi\in\bE(b)\;\big|\;
   \sup_{t>0}|\Psi_{-t}(\xi)|<\infty\,\big\}.
  \notag
  \end{align}
  Observe that if $\Psi$ is quasi-hyperbolic then
  $$
  E^s(b)\cap E^u(b)=\0\qquad \text{ for all }\quad b\in\bB.
  $$
  
  \begin{Lemma}\label{qh}
  If $\Psi$ is a quasi-hyperbolic action then there is $\tau>0$ such
  that for all $b\in\bB$,
  $$
  \lV \Psi_\tau|_{E^s(b)}\rV<\tfrac 12,\qquad
  \lV \Psi_{-\tau}|_{E^u(b)}\rV<\tfrac 12.
  $$
 
  \begin{Corollary}\label{qhcor}
  A quasi-hyperbolic action $\Psi$ is hyperbolic if and only if
  $\bE=E^s+E^u$.
  \end{Corollary}
  
  \end{Lemma}

   \noindent{\bf Proof of Lemma~\ref{qh}:}
 
 Suppose that the lemma is false for $E^s$. 
 The proof for $E^u$ is similar.
 Then there exist a sequence 
 $v_n\in E^s$, $\lv v_n\rv=1$ such that
 $\lv \Psi_n(v_n)\rv \ge \frac 12$ for all $n\in\na$.
 We claim that there exists $C$ such that
 $$
 \sup_{b\in \bB}\; \sup_{t\ge 0} 
 \lV \Psi_t  \vert_{E^s(b)}\rV < C < +\infty\, .
 $$

 Suppose the claim is true. Let $w_n:= \Psi_n(v_n)$ and
 $y_n:=\pi(v_n)$. We have that $\frac 12 \le \lv w_n\rv \le C$
 for all $n$ and
 $$
 \lv \Psi_t(w_n)\rv = \lv \Psi_{t+n}(v_n)\rv
 \le C \qquad \text{for all } t\ge -n\, .
 $$
 Since $\bB$ is compact there exists a convergent subsequence
 of $(y_n,w_n)$.
 If $(y,w)=\lim_n(y_n,w_n)$,
 we have that $y\in \bB$, $w\in \bE(y)$, $\lv w \rv \ge \frac 12$ and
 $$
 \lv \Psi_t(w)\rv \le C
 \qquad \text{for all } t\in {\mathbb \re}\, .
 $$
 This contradicts the quasi-hyperbolicity of $\Psi$.
 
 Now we prove the claim. Suppose it is false.
 Then there exist $t_n\ge 0$, $v_n\in E^s$,
 $\lv v_n\rv = 1$, such that
 \begin{equation}\label{E:QA.1}
 \sup_n \lv \Psi_{t_n} (v_n) \rv = +\infty\, .
 \end{equation}
 Let $s_n>0$ be such that
 $$
 \lv \Psi_{s_n}(v_n) \rv
 > \tfrac 12 \, \sup_{s\ge 0}\lv \Psi_s(v_n) \rv
 \ge \tfrac 12 \, \lv \Psi_{t_n}(v_n) \rv\, .
 $$
 By \eqref{E:QA.1} we have that $s_n\underset{n}{\to} +\infty$. Let
 $y_n:=\pi(\Psi_{s_n}v_n)$ and
 $$
 w_n := \frac{\Psi_{s_n}(v_n)}
             {\lv \Psi_{s_n}(v_n)\rv}
 \, .
 $$
 Then $\lv w_n\rv =1$ and if $t>-s_n$ we have that
 $$
 \lv \Psi_t(w_n)\rv
 = \frac{\lv \Psi_{t+s_n}(v_n) \rv}
        {\lv \Psi_{s_n}(v_n) \rv}
 \le \frac{2\, \lv \Psi_{t+s_n}(v_n)\rv}
          {\sup_{s\ge 0}\lv \Psi_s(v_n)\rv}
 \le 2\, .
 $$
 Since $\lv w_n\rv=1$ and $y_n\in \bB$, there exists a convergent
 subsequence $(y_n,w_n)\to (y,w)$. We would have that $y\in \bB$,
 $w\in \bE(y)$, $\lv w\rv=1$ and
 $$
 \lv \Psi_t(w)\rv \le 2
 \qquad \text{for all } t\in\re \, .
 $$
 This contradicts the quasi-hyperbolicity of $\Psi$.
 
 \qed

 \begin{Proposition}\label{Ybd}\quad
 
  Let $\La\subset T^*M\times\re$ and $H\in\cH(T^*M\times\re, b_1,b_2)$
  be a bounded hamiltonian. Then there is $b>0$ such that if $T>1$ and
  the segment $\psi_{[-1,T+2]}(\th)\subset\La$ is disconjugate, then
  $$
  |Y_\th(t)\;v|>b\;|v|\qquad\text{ for all }
  1<|t|\le T,\quad v\in\V(\th),
  $$
  where $Y_\th(t)\;v:=d\pi\big(d_\th\psi_t(0,v)\big)$.
 \end{Proposition}

 \begin{proof}
 
 It is enough to prove it for $t=T$. For $\vrt\in T^*M$ and $v\in
 T_{\pi(\vrt)} M$, write 
 \linebreak
 $Y_\vrt(t)\;v:=d\pi\big(d_\vrt\psi_t(0,v)\big)$,
 where $(0,v)\in\V(\vrt)$.
 
 Let $v_0\in T_{\pi(\th)} M$ with $|v_0|=1$. Let $J:[-1,T+2]\to TM$ be the
 vector field along $t\mapsto\pi\circ\psi_t(\th)$ defined as $J(t)=0$
 if $t\in[-1,0]$, $J(t)=Y_\th(t)\,v_0$ if $t\in[0,T]$ and
 $J(t)=Y_{\psi_{T+2}(\th)}\big(t-(T+2)\big)\,w_0$ if $t\in[T,T+2]$,
 where $w_0\in T_{\pi\psi_{T+2}(\th)}M$ is such that $J(t)$ is
 continuous. 
 
 Using proposition~\ref{Iform} and proposition~\ref{BdRic},
  $$
 |I(J,J)|=| J(T)^* (S_2-S_1)\,J(T) |
 \le 2\;A\;|J(T)|^2,
 $$
 where $S_1(t),\,S_2(t):\H(\psi_t(\th))\to\V(\psi_t(\th))$ are
 solutions of the Riccati equation corresponding to the subspaces
 $t\mapsto d_\th\psi_t(\V(\th))$ and $t\mapsto
 d\psi_{t-(T+2)}(\V(\psi_{T+2}(\th)))$.
 
 Fix a smooth function $f:[-1,1]\to[0,1]$ with support in $[-\frac
 12,\frac 12]$ such that $f(0)=1$. Let $P_t:T_{\pi(\th)}M\to
 T_{\pi(\psi_t(\th))}M$ be the parallel transport along
 $t\mapsto\pi\circ\psi_t(\th)$. Let $Z(t)=f(t)\,P_t(v_0)$ for
 $t\in[-1,1]$ and $Z(t)=0$ if $|t|>1$. Using formula~\eqref{index2} we
 have that there is $B=B(b_1,b_2,f)>0$ such that
 $$
 |I(Z,Z)|\,\le\, 4\;\lV L\rV_{C^2}\,\Vert(Z,\dot Z)\Vert^2_{\cL^2}\,\le\, B,
 $$
 where $L$ is the lagrangian of $H$.
 
 Using proposition~\ref{Iform},
 $$
 I(J,Z)=I(Z,J)= Z(0^-)\cdot 0-Z(0^+)\cdot v_0=-|v_0|^2=-1.
 $$
 
 Since the segment $\psi_{[-1,T+2]}(\th)$ is disconjugate, the
 index
 $$
 I(J+\la Z,\,J+\la Z)=I(J,J)+2 \la\,I(J,Z)+\la^2\, I(Z,Z)
 $$
 is positive for all $\la\in\re$. This polynomial in $\la$ has no real
 roots and hence
 $$
 4- 8\,A\,|J(T)|^2\,B\;\le\text{ discriminant }<\;0.
 $$
 $$
 |Y_\th(T)\,v_0|=|J(T)|\ge(2 A B)^{-\frac 12}\;|v_0|.
 $$
 \end{proof}

 Now let $\Si$ be a compact regular energy level of an autonomous
 hamiltonian. Recall that
 $$
 \N(\th)=\{\,\xi\in T_\th\Si\;|\;\langle d\pi(\xi),d\pi (X(\th))\rangle
 =0\,\},
 $$
 that $\P(\th):T_\th\Si\to\N(\th)$ is the projection along the flow
 direction and that if the orbit $\psi_\re(\th)$ is disconjugate
 then $\E^\top=\E\cap\N$ and  $\F^\top=\F\cap\N$ along $\psi_\re(\th)$.
 Let $\Psi:\re\to\isom(\N)$ be the action $\Psi_t=\P\circ d\psi_t$.
 Then orbitwise $\Psi_t$ coincides with the derivative $d\phi_t$ of the
 reparametrization $\phi_t$ defined in section~\ref{Saut}.

 \bigskip

 \noindent{\bf Proof of Theorem~\ref{TC}:}
 
 Since $\E^\top$ and $\F^\top$ are lagrangian subspaces of $\N$,
 $\dim \E^\top=\dim\F^\top=\frac 12\dim \N$. Proposition~\ref{Xhyp} and
 Corollary~\ref{qhcor}
 together with the following Proposition~\ref{hyp} imply
 Theorem~\ref{TC}.
 
 \qed

 \begin{Proposition}\label{hyp}
 Let $H:T^*M\to\re$ be a convex hamiltonian and let 
 $\La\subset T^*M$ be a compact invariant subset whose orbits
 are disconjugate.
 For $\th\in\La$ write
 \begin{align*}
 E^s(\th):&=\big\{\,\xi\in \N(\th)\;|\;\sup_{t>0}|\Psi_t(\xi)|<\infty\,\}
 \\
 E^u(\th):&=\big\{\,\xi\in \N(\th)\;|\;\sup_{t>0}|\Psi_{-t}(\xi)|<\infty\,\}.
 \end{align*}
 If $\E(\th)\cap\F(\th)=\langle X(\th)\rangle$ for all $\th\in\La$
 then
 \begin{enumerate}
 \item\label{h1} The action $\Psi$ over $\La$ is quasi-hyperbolic.
 \item\label{h2}
  $\E^\top(\th)\subset E^s(\th)$ and $\F^\top(\th)\subset E^u(\th)$
 for all $\th\in\La$.
 \end{enumerate}
 \end{Proposition}
 
 \begin{proof}\quad
 
  \eqref{h1}. 
  Let
  $
  \B^\top:=\{\, \xi\in\N\;|\; \sup_{t\in\re}|\,\Psi_t(\xi)|<+\infty\,\}$.
  By Proposition~\ref{growth}.\eqref{g2} applied to the action $d\phi_t\vert_\N=\Psi_t$
  we have that  $\B^\top\subset
  \E^\top\cap\F^\top=\langle X\rangle\cap\N=\0$.
 
  \eqref{h2}. 
  We only prove that $\E^\top\subset E^s$, the other inclusion is
  similar. Suppose it is false, then there is $\th\in\La$ and
  $\xi\in \E^\top(\th)\setminus E^s(\th)$. Since $\xi\notin E^s(\th)$, 
  $\sup_{t>0}|\Psi_t(\xi)|=\infty$.
  
  Write $\V^\top:=\V\cap\N$, $\H^\top:=\H\cap\N$. 
  Observe that $\N=\H^\top\oplus\V^\top$.
  Let $\xi_n\in E_n:=\Psi_{-n}(\V^\top(\psi_n(\th)))$ be such that
  $d\pi(\xi_n)=d\pi(\xi)$. Since $\lim_n E_n=\E^\top(\th)$ and 
  $\E^\top\pitchfork\V^\top$, we have that $\lim_n\xi_n=\xi$.
  Then 
  $$
  \lim_n\sup_{t\in[0,n]}|\Psi_t(\xi_n)|=\infty.
  $$

  Let $s_n\in[0,n]$ be such that $|\Psi_{s_n}(\xi_n)|=\sup_{t\in[0,n]}|\Psi_t(\xi_n)|$.
  Since $|\xi_n|$ is bounded on $n$, $\lim_n s_n=+\infty$.
  Define $w_n$ by  $(0,w_n)=\Psi_n(\xi_n)\in\V^\top$. 
  By Proposition~\ref{Ybd}, 
  $$
  |d\pi(\xi)|= |d\pi(\xi_n)|=\lv Y_{\psi_n(\th)}(-n)\,w_n\rv\ge b\;|w_n|.
  $$
  Hence $|\Psi_n(\xi_n)|=|w_n|$ is bounded on $n$.
  Therefore $\lim_n (n-s_n)=+\infty$.
  
   Let 
  $$
  \zeta_n:=\frac{\Psi_{s_n}(\xi_n)}{|\Psi_{s_n}(\xi_n)|}
  \qquad  \text{ and }\qquad
  \th_n:=\psi_{s_n}(\th).
  $$
  Then 
  $$
  |\Psi_t(\zeta_n)|\le 1\qquad\text{ for all }\quad|t|\le\min\{s_n,
  n-s_n\}.
  $$
    Since $\La$ is compact, there is a subsequence $n_k\to+\infty$
  such that the limits $\th_0:=\lim_k\th_{n_k}\in\La$ and
  $\zeta_0:=\lim_k\zeta_{n_k}\in T_{\th_0}T^*M$ exist.
  Since $\lim_n s_n=+\infty=\lim_n(n-s_n)$, we have that
  $|\zeta_0|=1$ and 
  $$
  |\Psi_t(\zeta_0)|\le 1 \qquad\text{ for all }\quad t\in\re.
  $$
  Then $\zeta_0\in\B^\top(\th_0)\setminus\0$. This contradicts
  item~\eqref{h1}.
 
  \end{proof}

 \bigskip

 \begin{Lemma}\label{EFNX}
 If $\th$ has no conjugate points
 $$
 \E(\th)\cap \F(\th)\cap \N(\th)=\{0\}
 \qquad \iff\qquad
 \E(\th)\cap\F(\th)=\langle X(\th)\rangle.
 $$
 \end{Lemma}
 \begin{proof}\quad

 We only prove ($\Rightarrow$). Suppose that 
 $\xi\in \E(\th)\cap \F(\th)$ and $E(\th)\cap \F(\th)\cap \N(\th)=\{0\}$.
 By \ref{growth}.\eqref{g4},
 $d\pi(X(\th))\ne 0$.
 Then there is  $\la\in\re$ satisfying 
 $$
 \langle d\pi(\xi)-\la\, d\pi(X(\th)),d\pi(X(\th))\rangle_{\pi(\th)} =0.
 $$
 Thus $\xi-\la\, X(\th)\in \N(\th)$.
By \ref{growth}.\eqref{g3}, $X(\th)\in \E(\th)\cap \F(\th)$.
Then 
$$
\xi-\la \, X(\th)\in \E(\th)\cap \F(\th)\cap \N(\th)=\{0\}.
$$
 Thus $\xi=\la\, X(\th)\in \langle X(\th)\rangle$.

 \end{proof}

 \noindent{\bf Proof of Theorem~\ref{TA}:}
 
 Let $H$ be the hamiltonian of $L$. We also denote by $\La$ the
 corresponding compact invariant subset for the hamiltonian flow.
 By lemma~\ref{I+}, the orbits in $\La$ are disconjugate.
 By \ref{growth}.\eqref{g4}, $d\pi(X(\th))\ne 0$ for all $\th\in \La$. For $\th\in\La$ write
 $$
 \cN(\th)=\{\,h\in T_\th M\;|\;\langle h,d\pi\,X(\th)\rangle =0\,\},
 $$
 Then $\cN(\th)=d\pi\,\N(\th)$. Fix $\th\in\La$. Let $h\in\cN(\th)$
 with $|h|=1$. Given $T>0$ let $(0,v_{-T})\in\V(\psi_{-T}(\th))$ and
 $(0,v_T)\in\V(\psi_T(\th))$ be such that
 $d\pi\,d\psi_T(0,v_{-T})=h=d\pi\,d\psi_{-T}(0,v_T)$. Let
 $$
 \xi_T(t):=\begin{cases}
 d\pi\,d\psi_{t+T}(0,v_{-T}) &\text{ if }\quad t\in[-T,0], \\
 d\pi\,d\psi_{t-T}(0,v_{-T}) &\text{ if }\quad t\in[0,T].
           \end{cases}
 $$
 Then $\xi_T\in\Om(\psi_{-T}(\th),2T)$.
 Also, there is $0<b<a$ such that
  \begin{equation}\label{|h|}
 \lV\xi_T\rV_{\cL^2}^2\ge b\,|h|^2\qquad \text{ for all }\quad T>1.
 \end{equation}
 Let $S_T,\,U_T:\H(\th)\to\V(\th)$ be the linear maps given by
 \begin{align*}
 E_T(\th)&=d\psi_{-T}(\V(\psi_{T}(\th)))=\graph(S_T),\\
 F_T(\th)&=d\psi_{T}(\V(\psi_{-T}(\th)))=\graph(U_T).
  \end{align*}
 From proposition~\ref{Iform} and~\eqref{|h|}, we have that
 $$
 b\,|h|^2\le dA_L(\xi,\xi)=h^*(U_T-S_T)\,h.
 $$
 Since $\lim_{T\to+\infty}S_T=S$, $\lim_{T\to+\infty}U_T=U$, where
 $\E(\th)=\graph(S)$, $\F(\th)=\graph(U)$; we have that 
 $(U-S)|_\cN\succ b\,I\succ 0$. This implies that
 $\E(\th)\cap\F(\th)\cap\N(\th)=\0$. By lemma~\ref{EFNX},
 $\E(\th)\cap\F(\th)=\langle X(\th)\rangle$ for all $\th\in\La$.
 By Theorem~\ref{TC}, $\psi_t|_\La$ is hyperbolic.
 
 \qed

  \begin{appendix}
    
  \section{Quasi Hyperbolicity}
  \label{apa}

  For completeness we prove
  
  \begin{Proposition}\label{PA1}
If $\Psi$ is a quasi-hyperbolic action over a compact base $\bB$ then:
\begin{enumerate}
\item $\Psi$ is hyperbolic over the non-wandering set $\Om(\psi|_B)$ of $\bB$. 
\item If in \eqref{Es} $\dim E^s$ is constant on all the minimal sets of $\bB$,
         then $\Psi$ is hyperbolic on all $\bE$.
\end{enumerate}
  \end{Proposition}
  
  \begin{Lemma}\label{33}
 \cite[L.~3.3, p.~928]{CI}.
  
  There exists $K>0$ such that for all $x\in \bB$ and $v\in\bE(x)$,
  $$
  \lv\Psi_t(v)\rv\le K(|v|+|\Psi_s(v)|)
  \quad \text{for all }
  0\le t\le s.
  $$
  \end{Lemma}
  
  \begin{proof}
  Suppose it is false.
  Then there exist $x_n\in \bB$, $0\ne v_n\in \bE(x_n)$  and $ 0\le t_n\le s_n$, such that 
  $$
  \lv\Psi_{t_n}(v_n)\rv
  \ge n \, (|v_n|+|\Psi_{s_n}(v_n)|).
  $$
  Then $t_n\to+\infty$ and $s_n-t_n\to+\infty$.
  We can assume that $|\Psi_{t_n}(v_n)|=\sup\limits_{0\le t\le s_n}|\Psi_t(v_n)|$.
  Let
  $$
  w_n:=\frac{\Psi_{t_n}(v_n)}{|\Psi_{t_n}(v_n)|}.
  $$
  For $-t_n<t<s_n-t_n$ we have that 
  $$
 | \Psi_t(w_n)|=\frac{|\Psi_{t+t_n}(v_n)|}{|\Psi_{t_n}(v_n)|}
 \le \frac1{|\Psi_{t_n}(v_n)|} \sup_{0\le t\le s_n}|\Psi_t(v_n)|=1.
  $$
  Taking a subsequence if necessary, we can assume that 
  $x_n\to x\in \bB$, $w_n\to w\in\bE(x)$.
  Then $|\Psi_t(w_n)|\le 1$ for all $t\in\re$ with $|w|=1$.
  This contradicts the quasi-hyperbolicity of $\Psi$.
  
  \end{proof}

  For $x\in \bB$, let $\a(x)$ and $\om(x)$ be the $\a$-limit of $x$ and the $\om$-limit of 
  $x$ respectively:
  \begin{align*}
  \a(x):&=\{\; \lim_n \psi_{s_n}(x)\;|\; s_n\to -\infty\;\},
  \\
  \om(x):&=\{\; \lim_n \psi_{t_n}(x)\;|\; t_n\to +\infty\;\}.
  \end{align*}

  The following Lemma appears in Sacker and Sell II \cite{SaSe}, Lemma~8, p. 485.
  \begin{Lemma}
  If $x\in \bB$, $a\in \a(x)$, $b\in \om(x)$ then
  \begin{align}
  \dim E^u(b) &\ge \dim \bE - \dim E^s (x),
  \label{dub}\\
  \dim E^s(a) &\ge \dim \bE - \dim E^u(x).
  \label{dsb}
  \end{align}

  \end{Lemma}
  
  \begin{proof}
  We only prove \eqref{dub}, inequality \eqref{dsb} is similar.
  Let $F\subset \bE(x)$ be a subspace such that $F\oplus E^s(x) = \bE(x)$.
  Let $s_n\ge 0$ be such that $\lim_n \psi_{s_n}(x) = b$ and $\lim_n s_n =+\infty$.
  Taking a subsequence we can assume that the limit $\F:=\lim \Psi_{s_n}(F)\subset \bE(b)$
  exists. 
  We will prove that $\F\subset E^u(b)$.
  Since $\dim\F=\dim F= \dim\bE -\dim E^s(x)$, we obtain \eqref{dub}.
  
  Let $w\in \F$ with $|w|=1$. Let $w_n\in \psi_{s_n}(F)$ with $|w_n|=1$ be such that 
  $\lim_n w_n =w$. By Lemma~\ref{33} we have that
  \begin{equation}\label{eq35}
  \big|\Psi_{t-s_n}(w_n)\big| \le K\;\big( |\Psi_{-s_n}(w_n)| + |w_n| \big)
  \qquad \text{ for all }\quad 0\le t \le s_n.
  \end{equation}
  Let
  $$
  v_n:=\frac{\Psi_{-s_n}(w_n)}{|\Psi_{-s_n}(w_n)|} \in F.
  $$
  Then
  \begin{equation}\label{eq36}
  \big|\Psi_{t}(v_n)\big|
  \le K \left(|v_n|+\frac{|w_n|}{|\Psi_{-s_n}(w_n)|}\right)
  \qquad\text{for all } \quad 0\le t\le s_n.
  \end{equation}
  Take a subsequence such that $v=\lim_n v_n\in F\setminus\{0\}$ exists, $|v|=1$.
  We claim that 
  \begin{equation}\label{lpsn}
  \lim_n |\Psi_{-s_n}(w_n)|=0,
  \end{equation}
  because if $\liminf_n|\Psi_{-s_n}(w_n)|\ge A>0$, from \eqref{eq36} we would have
  that 
  $$
  |\Psi_t(v)| \le K \big( 1+\tfrac 1A\big) \qquad \text{for all }\quad t\ge 0.
  $$
  Then $v\in E^s(x)$, which contradicts $v\in F\setminus\{0\}$.
  
  From \eqref{lpsn} and \eqref{eq35}
  $$
  |\Psi_{-r}(w_n)| \le 2 K \qquad \text{for all }\quad 0\le r\le s_n.
  $$
  This implies that $w=\lim_n w_n\in E^u(b)$.
  Therefore $\F\subset E^u(b)$.
    
  \end{proof}

\begin{Lemma}\label{lem_a4}
There are $A>0$ and $a>0$ such that 
\begin{align*}
|\Psi_t(v)|&\le A\,|v|\, \exp(-at) \qquad \text{for all } x\in B,\; v\in E^s(x) \text{ and } t\ge 0,
\\
|\Psi_{-t}(v)|&\le A\,|v|\, \exp(-at) \qquad \text{for all } x\in B,\; v\in E^u(x) \text{ and } t\ge 0.
\end{align*}
 
\end{Lemma}

\begin{proof}
 From Lemma~\ref{qh} we have that 
if $0\le s\le \tau$ and $n\in \na$,
$$
\big\Vert\psi_{n\tau+s}\big |_{E^s(b)}\big\Vert \le B\, 2^{-n},
\qquad
\big\Vert\psi_{-n\tau-s}\big |_{E^u(b)}\big\Vert \le B\, 2^{-n},
$$
where  $B:=\sup_{|t|\le \tau} \lV \Psi_t\rV$.
These imply Lemma~\ref{lem_a4}.
\end{proof}

\begin{Corollary}\label{CA5}
The functions $\dim E^s(x)$ and $\dim E^u(x)$ on $B$ are upper semicontinuous.

Also if $a\in\a(x)\cup\om(x)$ then $\dim E^s(a) \ge \dim E^s(x)$
and  $\dim E^u(a) \ge \dim E^u(x)$.
\end{Corollary}
\begin{proof}\quad 

From Lemma~\ref{lem_a4}  if $\lim_n x_n=x$ then $\limsup_n E^s(x_n)\subset E^s(x)$ and
$\limsup_n E^u(x_n)\subset E^u(x)$.  Hence $\dim E^s(x)\ge \limsup_n \dim E^s(x_n)$
and  $\dim E^u(x)\ge \limsup_n \dim E^u(x_n)$.
If $a\in \a(x)\cup \om(x)$, take $|s_n|\to \infty$ such that $\lim_n \psi_{s_n}(x)=a$ and 
apply the previous argument to $x_n:=\psi_{s_n}(x)$.
\end{proof}

{\bf Proof of Proposition~\ref{PA1}:}
  
  By Corollary~\ref{qhcor} to prove the hyperbolicity of $\Psi$
  it is enough to prove that $\bE=E^s+E^u$.

    (i). \cite[p. 929]{CI}.
    Given $x\in\Om(\psi|_\bB)$ there are $x_n\to x$, $s_n\to+\infty$ with
    $\psi_{s_n}(x_n)\to x$. Let $E_n$ be a subspace of $\bE(x_n)$ such that 
    $$
    \dim E_n= \dim\bE(x)-\dim E^s(x),
    \qquad
    E^s(x)\oplus\lim_n E_n=\bE(x).
    $$
    We claim that 
    \begin{equation}\label{cCt}
    \exists C>0 \quad \forall t\in[0,s_n]\qquad
    \lV \Psi_{-t}|_{\Psi_{s_n}(E_n)}\rV\le C.
    \end{equation}
    Using~\eqref{cCt},
    $\limsup_n\Psi_{s_n}(E_n)\subset E^u (x)$.
    Then $\dim E^u(x)+\dim E^s(x)\ge \dim E_n + \dim E^s(x)=\dim \bE(x)$.
    Thus $\bE(x)=E^s(x)+E^u(x)$.
    
    Now we prove~\eqref{cCt}.
    Suppose that~\eqref{cCt} is false. Then
    $$
    \forall n\in\na \quad \exists v_n\in \Psi_{s_n}(E_n) \quad \exists t_n\in[0,s_n]\qquad
    |v_n|= 1 \qquad |\Psi_{-t_n}(v_n)|\ge n.
    $$
    By lemma~\ref{33},
    \begin{gather*}
    n\le |\Psi_{-t_n}(v_n)|\le K\,(|\Psi_{-s_n}(v_n)|+|v_n|).
    \end{gather*}
    Then 
    $|\Psi_{-s_n}(v_n)|\ge \tfrac{n-K}K$.
    Also from lemma~\ref{33},
    $$
    \forall t\in[0,s_n]\qquad
    \frac{|\Psi_t(\Psi_{-s_n}(v_n))|}{|\Psi_{-s_n}(v_n)|}
    \le K + \frac{K}{|\Psi_{-s_n}(v_n)|}.
    $$
    Let $w_n:=\Psi_{-s_n}(v_n)/|\Psi_{-s_n}(v_n)|\in E_n$.
    The estimates above give
    $$
    \forall t\in [0,s_n]\qquad
    |\Psi_t(w_n)|\le K+ \tfrac{K^2}{n-K}.
    $$
    If $w_n\to w$, then $|w|=1$, $w\in\lim_n E_n$. Thus $w\notin E^s(x)$.
    But $|\Psi_t(w)|\le K$ for all $t\ge 0$.
    This is a contradiction.

    (ii). Let $x\in B$ and $a\in\a(x)$, $b\in\om(x)$ be such that $a$ and $b$ are in minimal sets.
    By item \ref{PA1}.(i) and \eqref{dub}, we have that
    $$
    \dim \bE -\dim E^s(b)=\dim E^u(b) \ge \dim\bE-\dim E^s(x).
    $$
    Thus $\dim E^s(x)\ge \dim E^s(b)$. By Corollary~\ref{CA5}
    $\dim E^s(b)\ge \dim E^s(x)$. Therefore
    $$
    \dim E^s(b)=\dim E^s(x).
    $$
    Similarly
       $$
    \dim E^u(a)=\dim E^u(x).
    $$
    By the hypothesis on the minimal sets in item (ii), 
    $\dim E^u(a) = \dim E^u(b)$.
    By item~\ref{PA1}.(i) $\dim E^u(b) = \dim\bE-\dim E^s(b)$.
    Therefore
    $$
    \dim E^u (x) = \dim E^u (a) =\dim E^u(b) =\dim \bE-\dim E^s(b) = \dim \bE -\dim E^s (x).
    $$
    Since by the quasi-hyperbolicity $E^s(x) \cap E^u(x)=\{0\}$,
    we obtain that 
    $\bE(x)=E^s (x) + E^u (x)$ for all $x\in B$.
    Then Corollary~\ref{qhcor} implies that $\Psi|_B$ is hyperbolic.

    \qed

  \end{appendix}


\def\cprime{$'$} \def\cprime{$'$} \def\cprime{$'$} \def\cprime{$'$}
\providecommand{\bysame}{\leavevmode\hbox to3em{\hrulefill}\thinspace}
\providecommand{\MR}{\relax\ifhmode\unskip\space\fi MR }
\providecommand{\MRhref}[2]{%
  \href{http://www.ams.org/mathscinet-getitem?mr=#1}{#2}
}
\providecommand{\href}[2]{#2}

\end{document}